\documentclass{amsart}
\usepackage{mathrsfs}
\usepackage[all]{xy}

\usepackage{wasysym}
\usepackage{amssymb}

\usepackage{ifpdf}
\ifpdf 
  \usepackage[pdftex]{graphicx}
  \DeclareGraphicsExtensions{.pdf,.png,.jpg,.jpeg,.mps}
  \usepackage{pgf}
\else 
  \usepackage{graphicx}
  \DeclareGraphicsExtensions{.eps,.bmp}
  \usepackage{pgf}
  \usepackage{pstricks}
\fi
\usepackage{epic,bez123}
\usepackage{wrapfig}

\newtheorem{thm}{Theorem}[section]
\newtheorem{prop}[thm]{Proposition}
\newtheorem{cor}[thm]{Corollary}
\newtheorem{lem}[thm]{Lemma}

\newtheorem*{claim1}{Claim 1}
\newtheorem*{claim2}{Claim 2}

\newtheorem*{conjA}{Conjecture A}
\newtheorem*{conjB}{Conjecture B}

\newtheorem{assump}{}

\newtheorem{conv}{Convention}

\theoremstyle{definition}
\newtheorem{defn}[thm]{Definition}

\theoremstyle{remark}

\newtheorem{rem}[thm]{Remark}
\newtheorem{example}[thm]{Example}

\newtheorem*{ack}{Acknowledgment}

\newcommand{\G }{\mathscr{G} (G, X\cup \mathcal H)}
\newcommand{\K }{\mathscr{G} (\Gamma, Y\cup \mathcal K)}

\newcommand{\Ghat }{\mathscr{G} (G, \hat X \cup \mathcal H)}
\newcommand{\Khat }{\mathscr{G} (\Gamma, Y\cup \mathcal H_\Gamma)}
\newcommand{\Gphat }{\mathscr{G} (G, \hat X \cup \mathcal P)}

\newcommand{\Gp }{\mathscr{G} (G, X\cup \mathcal P)}

\newcommand{\Ga }{\mathscr{G} (G, \mathcal A)}
\newcommand{\Gx }{\mathscr{G} (G, X)}

\newcommand{\GH }{(G, \mathbb H)}
\newcommand{\KK }{(\Gamma, \mathbb K)}
\newcommand{\GP }{(G, \mathbb P)}

\newcommand{\dxh }{d_{X\cup\mathcal H}}
\newcommand{\dxp }{d_{X\cup\mathcal P}}

\newcommand{\dxo }{d_{X\cup{\Omega}}}

\newcommand{\da }{d_{\mathcal A}}

\newcommand{\Hl }{\{ H_i \}_{i \in I} }
\newcommand{\Kl }{\{ K_j \}_{j \in J} }
\newcommand{\Pl }{\{ P_j \}_{j \in J} }

\newcommand{\Gf}{\overline{G}_f}
\newcommand{\pGf}{\partial_f{G}}

\newcommand{\TGH}{T_{\mathbb H}}
\newcommand{\TGP}{T_{\mathbb P}}

\newcommand{\LH}{\Lambda_{\mathbb H}}
\newcommand{\LP}{\Lambda_{\mathbb P}}
\newcommand{\LF}{\Lambda_{f}}

\newcommand{\lab }{\textbf{Lab}}

\newcommand{\len }{\ell}

\newcommand{\norm}[2]{\lvert #1 \rvert_{#2}}

\usepackage[bookmarks=true, pdfauthor={YANG Wenyuan}]{hyperref}

\begin{document}

\title{Peripheral structures of relatively hyperbolic groups}

\author{Wen-yuan Yang}

\address{College of Mathematics and Econometrics, Hunan
University, Changsha, Hunan 410082 People's Republic of China}
\curraddr{U.F.R. de Mathematiques, Universite de Lille 1, 59655
Villeneuve D'Ascq Cedex, France} \email{wyang@math.univ-lille1.fr}
\thanks{The author is supported by the China-funded Postgraduates
Studying Aboard Program for Building Top University. This research
was supported  by National Natural Science Foundational of China
(No. 11071059)}



\date{June 17, 2011}

\dedicatory{}


\begin{abstract}
In this paper, we introduce and characterize a class of
parabolically extended structures for relatively hyperbolic groups.
A characterization of relative quasiconvexity with respect to
parabolically extended structures is obtained using dynamical
methods. Some applications are discussed. The class of groups acting
geometrically finitely on Floyd boundaries turns out to be easily
understood. However, we also show that Dunwoody's inaccessible group
does not act geometrically finitely on its Floyd boundary.
\end{abstract}

\maketitle

\section{Introduction}
In this paper, we study peripheral structures of relatively
hyperbolic groups. Since introduced by Gromov \cite{Gro}, relative
hyperbolicity has several equivalent formulations. If relatively
hyperbolic groups are understood as geometrically finite actions(see
Bowditch \cite{Bow1}), peripheral structures can be  a set
of representatives of the conjugacy classes of maximal parabolic
subgroups. On the other hand, following approaches of Farb
\cite{Farb} and Osin \cite{Osin}, a peripheral structure is a
preferred collection of subgroups such that the constructed relative
Cayley graph satisfies some nice properties. In practice a given
group may be hyperbolic relative to different peripheral structures.
So it is interesting to understand the relationship between possible
peripheral structures that can be endowed on a countable group.

The study of peripheral structures actually stems from the study of hyperbolic
groups. A first result of this sort is due to Gersten \cite{Ger} and
Bowditch \cite{Bow1}, who proved that malnormal quasiconvex
subgroups of hyperbolic groups yield peripheral structures. In a
point of view of relative hyperbolicity, an ordinary hyperbolic
group is hyperbolic relative to a trivial subgroup. Later on, in
relatively hyperbolic groups, Osin \cite{Osin2} generalized their
results and proposed the notion of hyperbolically embedded
subgroups. A hyperbolically embedded subgroup can be added into the
existing peripheral structure such that the group is hyperbolic relative to the enlarged peripheral structure. In
the present paper, we shall enlarge peripheral subgroups themselves to
get new peripheral structures.

In the remainder of the paper, the term ``peripheral structure''
will be used in a weaker sense, i.e. it is just a finite collection of
subgroups without involving relative hyperbolicity.

Let $G$ be a countable group with a collection of subgroups $\mathbb{H}=\Hl$. We denote such a pair by $\GH$. The collection
$\mathbb H$ is often referred as a \textit{peripheral structure} of
$G$, and each element of $\mathbb H$ a \textit{peripheral subgroup}
of $G$. In what follows, we are interested in peripheral structures of finite cardinality.

Let $\mathbb H = \Hl$ and $\mathbb P = \Pl$ be two peripheral
structures of a countable group $G$. If for each $i \in I$, there
exists $j \in J$ such that $H_i \subset P_j$, then we say $\mathbb P$ is an \textit{extended peripheral structure} for the pair $\GH$.
Moreover, if the pairs $\GH$ and $\GP$ are both relatively hyperbolic, then we say
$\mathbb P$ is \textit{parabolically extended} for $\GH$. Each
subgroup $P \in \mathbb P$ is said to be \textit{parabolically embedded} into $\GH$.

Our first result is to give a characterization
of parabolically extended peripheral structure. The notation
$\Gamma^g$ denotes the conjugate $g \Gamma g^{-1}$ of a subgroup
$\Gamma \subset G$ by an element $g \in G$. Recall that a subgroup
$\Gamma \subset G$ is \textit{weakly malnormal} if $\Gamma \cap
\Gamma^g$ is finite for any $g \in G \setminus \Gamma$.

\begin{thm} \label{mainthm}
Suppose $\GH$ is relatively hyperbolic and $\mathbb P$ is an
extended peripheral structure for $\GH$. Then $(G,\mathbb P)$ is
relatively hyperbolic if and only if each $P \in \mathbb P$
satisfies the following statements

(P1). $P$ is relatively quasiconvex with respect to $\mathbb H$,

(P2). $P$ is weakly malnormal,  and

(P3). $P^g \cap P'$ is finite for any $g\in G$ and distinct $P , P' \in
\mathbb P$.
\end{thm}

In fact, Theorem \ref{mainthm} follows from a characterization of
parabolically embedded subgroups (see Theorem \ref{charpi}).

\begin{rem}
In our terms, a hyperbolically embedded subgroup $\Gamma \subset G$
in the sense of Osin \cite{Osin2} is parabolically embedded into $\GH$. In this case,
$\Gamma$ turns out to be a hyperbolic group.  However, parabolically embedded
subgroups may be in general hyperbolic relative to a nontrivial
collection of proper subgroups, as stated in Condition (P1).
\end{rem}

In relatively hyperbolic groups, we can define a natural class of
subgroups named relatively quasiconvex subgroups. We observe that a
subgroup relatively quasiconvex with respect to one peripheral
structure is not necessarily relatively quasiconvex with respect to others.
Our second result is to give a characterization of relative
quasiconvexity with respect to parabolically extended structures.

\begin{thm} \label{mainthm2}
Suppose $\GH$ is relatively hyperbolic and $\mathbb P$ is a
parabolically extended structure for $\GH$. If $\Gamma \subset G$ is
relatively quasiconvex with respect to $\mathbb H$, then $\Gamma$ is
relatively quasiconvex with respect to $\mathbb P$.

Conversely, suppose $\Gamma \subset G$ is relatively quasiconvex with
respect to $\mathbb P$. Then $\Gamma$ is relatively quasiconvex with
respect to $\mathbb H$ if and only if $\Gamma \cap P^g$ is
relatively quasiconvex with respect to $\mathbb H$ for any $g\in G$
and $P \in \mathbb P$.
\end{thm}
\begin{rem}
Theorem \ref{mainthm2} generalizes the main result of E.
Martinez-Pedroza \cite{MarPed}, where $\mathbb P$ is obtained from $\mathbb H$ by adjoining hyperbolically embedded subgroups.
\end{rem}

Unlike the proof of Theorem \ref{mainthm}, we give a dynamical proof
of Theorem \ref{mainthm2} using the work of Gerasimov \cite{Ge2} and
Gerasimov-Potyagailo \cite{GePo2} on Floyd maps.

Using Floyd maps, some preliminary observations are made to
general peripheral structures of relatively hyperbolic groups. In this study, it appears worth to distinguish relatively hyperbolic groups which admit geometrically finite actions on their Floyd boundaries and the ones which do not.  For instance, peripheral structures of groups admitting geometrically finite actions on Floyd boundaries are much simpler and shown to
be parabolically extended with respect to canonical ones. See
Corollary \ref{structure}.

In fact, it is known that many relatively hyperbolic groups act geometrically
finitely on their Floyd boundaries:

\begin{enumerate} \label{GFExmples}
\item[(1)] Geometrically finite Kleinian groups where maximal parabolic subgroups are
virtually abelian.
\item[(2)] Hyperbolic groups relative to a collection of
unconstricted subgroups. According to \cite{DruSapir}, a group is
unconstricted if one of its asymptotic cones has no cut points. By
Proposition 4.28 of \cite{OOS}, the Floyd boundary of an
unconstricted subgroup is trivial, i.e. consisting of less then 2 points.
\item[(3)] Most known hyperbolic groups relative to a collection of
non-relatively hyperbolic subgroups. Recall that an \textit{Non-Relatively Hyperbolic}
(NRH) group is not hyperbolic relative to any collection of proper
subgroups. For example, all NRH subgroups in \cite{AAS} have
trivial Floyd boundaries.
\end{enumerate}

However, there do exist relatively hyperbolic groups which do
not act geometrically finitely on their Floyd boundaries. See
Proposition \ref{Dunwoody} for the example Dunwoody's inaccessible
group which is constructed by Dunwoody in \cite{Dun}.

Moreover, one of our results shows that for a relatively
hyperbolic group, its convergence action on Floyd boundary is
largely determined by ones of peripheral subgroups.

\begin{thm} \label{mainthm4}
Suppose $\GH$ is relatively hyperbolic. Then $G$ acts geometrically
finitely on its Floyd boundary $\pGf$ if and only if each $H \in \mathbb H$
acts geometrically finitely on its limit set for the action on $\pGf$.
\end{thm}

In \cite{OOS}, Olshanskii-Osin-Sapir made the following conjecture on the
relationship between relatively hyperbolic groups and their Floyd boundaries.
\begin{conjA}\cite{OOS}
If a finitely generated group has non-trivial Floyd boundary, then
it is hyperbolic relative to a collection of proper subgroups.
\end{conjA}
\begin{rem}
The converse of Conjecture A follows from Gerasimov's theorem \cite{Ge1} on Floyd maps for relatively hyperbolic groups.
\end{rem}

By simple arguments, we observe that Conjecture A is in fact equivalent to the following cojecture. See Proposition \ref{conjecture}.

\begin{conjB}
If a finitely generated group is hyperbolic relative to a collection
of NRH proper subgroups, then it acts geometrically finitely on its
Floyd boundary.
\end{conjB}
\begin{rem}
One result of Behrstock-Drutu-Mosher \cite[Proposition 6.3]{BDM} says that Dunwoody's inaccessible group does not satisfy the assumption of Conjecture B.
\end{rem}


The structure of the paper is as follows. In Section 2, we restate
Bounded Coset Penetration property for countable relative
hyperbolicity and indicate the equivalence of the definitions of
countable relative hyperbolicity due to Osin and Farb. We also
construct a quasi-isometric map from a relatively quasiconvex
subgroup to the ambient group. Our construction leads to a new proof
of Hruska's result \cite{Hru} that relatively quasiconvex subgroups
are relatively hyperbolic. In Section 3, we study parabolically
extended structures and give the proof of Theorem \ref{mainthm}. In
Section 4, we take a dynamical approach to relative hyperbolicity
and then prove Theorems \ref{mainthm2} and \ref{mainthm4}.
Dunwoody's inaccessible group and Olshanskii-Osin-Sapir's conjecture are discussed in this section.

\begin{ack}
The author would like to sincerely thank Prof. Leonid Potyagailo for
many helpful comments on this work.  The author also thanks Denis
Osin for some corrections in an earlier version of Theorem
\ref{mainthm}, and Jason Manning for helpful conversations in a
Geometric Group Theory conference at Luminy in May 2010.
\end{ack}

\section{Preliminaries}

\subsection{Cayley graphs and partial distance functions}

Let $G$ be a group with a set $\mathcal A \subset G$. Note that the
alphabet set $\mathcal A$ is assumed to neither be finite and nor
generate $G$. For convenience, we always assume $1 \notin \mathcal
A$ and $\mathcal A = \mathcal A^{-1}$.

We define the \textit{Cayley graph} $\Ga$ of a group $G$ with
respect to $\mathcal A$, as a directed edge-labeled graph with the
vertex set $V(\Ga )=G$ and the edge set $E(\Ga )=G\times \mathcal
A$. An edge $e=[g,a]$ goes from the vertex $g$ to the vertex $ga$
and has the label $\lab (e)=a$. As usual, we denote the origin and
the terminus of the edge $e$, i.e., the vertices $g$ and $ga$, by
$e_-$ and $e_+$ respectively. By definition, we set $e^{-1} := [ga,
a^{-1}]$.

Let $p= e_1 e_2\ldots e_k$ be a combinatorial path in the Cayley
graph $\Ga $, where $e_1, e_2, \ldots , e_k\in E(\Ga )$. The length
of $p$ is the number of edges in $p$, i.e. $\len(p) = k$. We define
the label of $p$ as $\lab(p)=\lab(e_1)\lab(e_2)\ldots \lab (e_k).$
The path $p^{-1}$ is defined in a similar way. We also denote by
$p_-=(e_1)_-$ and $p_+=(e_k)_+$ the origin and the terminus of $p$
respectively. A \textit{cycle} $p$ is a path such that $p_- = p_+$.

\begin{defn}(Partial Distance Functions) \label{distance}
By assigning the length of each edge in $\Ga $ to be $1$, we define
a \textit{partial distance function} $\da: \Ga \times \Ga \to [0,
\infty]$ as follows. Note that $\mathcal A$ is not assumed to
generate $G$ and thus $\Ga$ may be disconnected. For $z, w \in \Ga$,
if $z$ and $w$ lie in the same path connected component of $\Ga$, we
define $\da(z,w)$ as the length of a shortest path in $\Ga$ between
$z$ and $w$. Otherwise we set $\da(z,w) = \infty$.
\end{defn}

\begin{rem} \label{wordmetric}
If $\langle \mathcal A \rangle = G$, then the partial distance
function $\da$ actually gives a \textit{word metric} with respect to
$\mathcal A$ on the Cayley graph $\Ga$. Note that if $g_1 , g_2 \in
G$ and $g_1^{-1} g_2 \notin \langle \mathcal A \rangle$, then
$\da(g_1, g_2) = \infty$. For any element $g \in G$, we define its
norm $\norm{ g }{\mathcal A} = \da(1, g)$.
\end{rem}

A path $p$ in the Cayley graph $\Ga$ is called \textit{$(\lambda,
c)$-quasigeodesic} for some $\lambda \geq 1$, $c \geq 0$, if the
following inequality holds for any subpath $q$ of $p$, $$\len(p)
\leq \lambda \da(q_-, q_+) + c.$$

We often consider a group $G$ with a collection of subgroups
$\mathbb H = \Hl$. Then $X$ is a \textit{relative generating set}
for $\GH$ if $G$ is generated by the set $\left(\cup_{i \in I} H_i
\right) \cup X$ in the traditional sense.

Let $\mathcal H= \bigsqcup\limits_{i \in I} H_i \setminus \{1\}$.
Fixing a relative generating set $X$ for $\GH$, the constructed
Cayley graph $\G$ is called the \textit{relative Cayley graph} of
$G$ with respect to $\mathbb H$. We now collect some notions
introduced by Osin \cite{Osin} in relative Cayley graphs.

\begin{defn}
Let $p, q$ be paths in $\G$. A subpath $s$ of $p$ is called an
\textit{$H_i $-component}, if $s$ is the maximal subpath of $p$ such
that $s$ is labeled by letters from $H_i$.

Two $H_i $-components $s, t$ of $p, q$ respectively are called
\textit{connected} if there exists a path $c$ in $\G$ such that
$c_-=s_-, c_+=t_-$ and $c$ is labeled by letters from $H_i$. An $H_i
$-component $s$ of $p$ is \textit{isolated} if no other
$H_i$-component of $p$ is connected to $s$.

We say a path $p$ \textit{without backtracking} by meaning that all
$H_i $-components of $p$ are isolated. A vertex $u$ of $p$ is
\textit{nonphase} if there is an $H_i$-component $s$ of $p$ such
that $u$ is a vertex of $s$ but $u \neq s_-, u \neq s_+$. Other
vertices of $p$ are called \textit{phase}.
\end{defn}

\subsection{Relatively hyperbolic groups}
In the large part of this paper, we consider countable relatively
hyperbolic groups. In this subsection, we shall recall the
definitions of countable relative hyperbolicity in the sense of Osin
and Farb, and then indicate their equivalence based on Osin's
results in \cite{Osin}.

Let $G$ be a countable group with a finite collection of subgroups
$\mathbb H = \Hl$. As the notion of relative generating sets, we can
define in a similar fashion the relative presentations and
(relative) Dehn functions of $G$ with respect to $\mathbb H$. We
refer the reader to \cite{Osin} for precise definitions.

We now give the first definition of relative hyperbolicity due to
Osin \cite{Osin}. Note that the full version of Osin's definition
applies to general groups without assuming the finiteness of
$\mathbb H$.

\begin{defn}\label{OsinRH} (Osin Definition)
A countable group $G$ is \textit{hyperbolic relative to $\mathbb H$
in the sense of Osin} if $G$ is finitely presented with respect to
$\mathbb H$ and the relative Dehn function of $G$ with respect to
$\mathbb H$ is linear.
\end{defn}

The following lemma plays an important role in Osin's approach
\cite{Osin} to relative hyperbolicity. The finite subset $\Omega$
and constant $\kappa$ below depend on the choice of finite relative
presentations of $G$ with respect to $\mathbb H$. In our later use
of Lemma \ref{Omega}, when saying there exists $\kappa, \Omega$ such
that the inequality (\ref{keyinequality}) below holds in $\G$, we have
implicitly chosen a finite relative presentation of $G$ with respect
to $\mathbb H$.

\begin{lem}\label{Omega} \cite[Lemma 2.27]{Osin}
Suppose $\GH$ is relatively hyperbolic in the sense of Osin and $X$
is a finite relative generating set for $\GH$. Then there exists
$\kappa \geq 1$ and a finite subset $\Omega \subset G$ such that the
following holds. Let $c$ be a cycle in $\G$ with a set of isolated
$H_i$-components $S=\{s_1, \ldots , s_k\}$ of $c$ for some $i \in
I$, Then
\begin{equation} \label{keyinequality}
\sum\limits_{s \in S} d_{\Omega_i}(s_-,s_+) \le \kappa \len(c),
\end{equation}
where $\Omega_i := \Omega \cap H_i$.
\end{lem}

\begin{rem}
By the definition of $d_{\Omega_i}$, if $d_{\Omega_i}(g,h) < \infty$
for $g, h \in G$, then there exists a path $p$ labeled by letters
from $\Omega_i$ in this new Cayley graph $\mathscr{G} (G, X \cup
\Omega \cup \mathcal H)$ such that $p_- = g, p_+ = h$.
\end{rem}

Using Lemma \ref{Omega}, the following lemma can be proven exactly
as Proposition 3.15 in \cite{Osin}. The finite set $\Omega$ below is
given by Lemma \ref{Omega}.
\begin{lem} \label{stableqg} \cite{Osin}
Suppose $\GH$ is relatively hyperbolic in the sense of Osin and $X$
is a finite relative generating set for $\GH$. For any $\lambda \ge
1, c \ge 0$, there exists a constant $\epsilon =
\epsilon(\lambda,c)>0$ such that the following holds. Let $p,q$ be
$(\lambda,c)$-quasigeodesics without backtracking in $\G$ such that
$p_-=q_-, p_+=q_+$. Then for any phase vertex $u$ of $p$(resp. $q$),
there exists a phase vertex $v$ of $q$(resp.$p$) such that
$\dxo(u,v) < \epsilon$.
\end{lem}

The following lemma is well-known in the theory of relatively
hyperbolic groups.

\begin{lem} \label{periphintersect} \cite{Osin}
Suppose $\GH$ is relatively hyperbolic in the sense of Osin. Then
the following statements hold for any $g \in G$ and $H_i, H_j \in
\mathbb H$, \\
1) If $H_i^g \cap H_i$ is infinite, then $g \in H_i$,\\
2) If $i \neq j$, then $H_i^g \cap H_j$ is finite.
\end{lem}

In order to formulate the BCP property, we shall put a metric $d_G$
on a group $G$, which is \textit{proper} if any bounded set is
finite, and \textit{left invariant} if $d_G(g x_1, g x_2) = d_G(x_1,
x_2)$ for any $g, x_1, x_2 \in G$. For given $g \in G$, we define
the norm $\norm{g}{d_G}$ with respect to $d_G$ to be the distance
$d_G(1,g)$.

Let us now recall the following lemma in Hruska-Wise \cite{HrWi},
which justifies the use of proper left invariant metrics on
countable groups.
\begin{lem} \cite{HrWi}
A group is countable if and only if it admits a proper left
invariant metric.
\end{lem}

From now on, we assume that $\GH$ has a finite relative generating
set $X$.

In terms of proper left invariant metrics, bounded coset penetration
property is formulated as follows.
\begin{defn}\label{BCPdef} (Bounded coset penetration)
Let $d_G$ be some (any) proper, left invariant metric on $G$. The
pair $\GH$ is said to satisfy the {\it bounded coset penetration
property with respect to $d_G$} (or BCP property with respect to
$d_G$ for short) if, for any $\lambda \geq 1$, $c \geq 0$, there
exists a constant $a = a(\lambda,c,d_G)$ such that the following
conditions hold. Let $p, q$ be $(\lambda, c)$-quasigeodesics
without backtracking in $\G$ such that $p_- = q_-$, $p_+ = q_+$.

1) Suppose that $s$ is an $H_i$-component of $p$ for some $H_i \in
\mathbb H$, such that $d_G(s_-, s_+) > a$. Then there exists an
$H_i$-component $t$ of $q$ such that $t$ is connected to $s$.

2) Suppose that $s$ and $t$ are connected $H_i$-components of $p$
and $q$ respectively, for some $H_i \in \mathbb H$. Then $d_G(s_-,
t_-) < a$ and $d_G(s_+, t_+) < a$.
\end{defn}

\begin{rem}
In \cite{Hru}, Hruska proposed to use the partial distance function
$d_X$ (with respect to a finite relative generating set $X$) instead
of $d_G$ in the definition of BCP property, and showed that BCP
property is independent of the choice of relative generating sets.
However, this is generally not true, due to the following example.
\end{rem}

\begin{example} \label{BCPexample}
We take a free product $G = L \ast_F K$ of two finitely generated
groups $L$ and $K$ amalgamated over a nontrivial finite group $F$,
which is known to be hyperbolic relative to $\mathbb H = \{L, K\}$
in the sense of Farb and Osin. Take a special relative generating
set $X = \emptyset$ and construct the relative Cayley graph $\G$.
Since $F = L \cap K$ is nontrivial, we take a nontrivial element $f
= f_L = f_K \in F$, where $f_L$ and $f_K$ are the corresponding
elements in $L \setminus \{1\}$ and $K \setminus \{1\}$
respectively. Thus, there are two different edges $p$ and $q$ with
same endpoints $1$ and $f$ such that $\lab(p) = f_L$ and $\lab(q) =
f_K$ in $\G$. Obviously $p$ and $q$ are geodesics and isolated
components. Note that $d_X(p_-,p_+) = \infty$. Hence BCP property is
not well-defined with respect to $d_X$.

This example was also known to other researchers in this field, see
Remark 2.15 in a latest version of \cite{MarPed}. Moreover, the idea
using proper metrics to define BCP property also appeared
independently in Martinez-Pedroza \cite{MarPed}. See Subsection 2.3
in \cite{MarPed}.
\end{example}

\begin{rem}
We remark that Hruska's arguments in \cite{Hru} remain valid with
the new definition \ref{BCPdef} of BCP property. So the main result
concerning the equivalence of various definitions of relative
hyperbolicity in \cite{Hru} is still correct. For the convenience of
the reader, we will give a direct proof of the equivalence of Osin's
and Farb's definitions in the remaining part of this subsection.
\end{rem}

The following corollary is immediate by an elementary argument.
\begin{cor} \label{BCPindep}
BCP property of $\GH$ is independent of the choice of left invariant
proper metrics.
\end{cor}

In view of Corollary \ref{BCPindep}, we shall not mention explicitly
proper left invariant metrics when saying the BCP property of $\GH$.

With a little abuse of terminology, we also say $\GH$ satisfies BCP
property with respect to a partial distance function $\da$ if, for
any $\lambda \geq 1$, $c \geq 0$, there exists a constant $a =
a(\lambda,c,\da)$ such that the statements of BCP property 1) and 2)
are true for $\da$.

When proving BCP property, we usually do it with respect to some
special partial distance function, as stated in the following
corollary.

\begin{cor} \label{BCPindep2}
Let $\mathcal A \subset G$ be a finite set. If $\GH$ satisfies BCP
property with respect to $\da$, then so does $\GH$ with respect to
any proper left invariant metric.
\end{cor}

The second definition of relative hyperbolicity is due to Farb
\cite{Farb}, which will be used in establishing relative
hyperbolicity of groups in Section 3.

\begin{defn}\label{FarbRH} (Farb Definition)
A countable group $G$ is \textit{hyperbolic relative to $\mathbb H$
in the sense of Farb} if the Caylay graph $\G$ is hyperbolic and the
pair $\GH$ satisfies the BCP property.
\end{defn}

As observed in Example \ref{BCPexample}, BCP property is not
well-defined with respect to relative generating sets. But the
following lemma states that for a given finite relative generating
set, we can always find a finite subset $\Sigma$ such that $\GH$
satisfies BCP property with respect to $d_\Sigma$.

\begin{lem} \label{BCPOmega}
Suppose $\GH$ is relatively hyperbolic in the sense of Osin and $X$
is a finite relative generating set for $\GH$.  Then there exists a
finite set $\Sigma \subset G$ such that then $\GH$ satisfies BCP
property with respect to $d_\Sigma$.
\end{lem}
\begin{proof}
Let $\Omega$ be the finite set given by Lemma \ref{Omega} for $\G$.
We take a new finite relative generating set $\hat X := X \cup
\Omega$. Using Lemma \ref{Omega} again, we obtain a finite set
$\Sigma$ and constant $\mu > 1$ such that the inequality
(\ref{keyinequality}) holds in $\Ghat$.

We now verify BCP property 1). Let $p, q$ be $(\lambda,
c)$-quasigeodesics without backtracking in $\G$. Since $\hat X$ is
finite, the embedding $\G \hookrightarrow \Ghat$ is a
quasi-isometry. Regarded as paths in $\Ghat$, $p,q$ are
$(\lambda',c')$-quasigeodesics without backtracking in $\Ghat$, for
some constants $\lambda' \ge 1, c'\ge 0$ depending on $\hat X$.

Let $\epsilon = \epsilon(\lambda,c)$ be the constant given by Lemma
\ref{stableqg}. Set $$a = \mu (\lambda' +1) (2 \epsilon + 1) +
c'\mu.$$

We claim that $a$ is the desired constant for the BCP property of
$\GH$. If not, we suppose there exists an $H_i$-component $s$ of $p$
such that $d_\Sigma(s_-, s_+)> a$ and no $H_i$-component of $q$ is
connected to $s$.

By Lemma \ref{stableqg}, there exist phase vertices $u, v$ of $q$
such that $\dxo(s_-,u) < \epsilon, \dxo(s_-,v) < \epsilon.$ Thus by
regarding $p, q$ as paths in $\Ghat$, there exist paths $l$ and $r$
labeled by letters from $\Omega$ such that $l_-=e_-,l_+=u,r_-=e_+,$
and $r_+=v$. We consider the cycle $c := e r [u, v]_q^{-1}l^{-1}$ in
$\Ghat$, where $[u,v]_q$ denotes the subpath of $q$ between $u$ and
$v$. Since $[u,v]_q$ is a $(\lambda',c')$-quasigeodesic, we compute
$\len(c)$ by the triangle inequality and have $$\len(c) \le
(\lambda' +1) (2 \epsilon + 1) + c'.$$ Obviously $e$ is an isolated
$H_i$-component of $c$. Using Lemma \ref{Omega} for the cycle $c$ in
$\Ghat$, we have $d_{\Sigma}(e_-,e_+) < \mu \len(c) < a$. This is a
contradiction.

Therefore, BCP property 1) is verified with respect to $d_\Sigma$.
BCP property 2) can be proven in a similar way.
\end{proof}

We conclude this subsection with the following theorem which is
proven in \cite{Osin} for finitely generated relatively hyperbolic
groups.
\begin{thm} \label{FarbOsin}
The pair $\GH$ is relatively hyperbolic in the sense of Farb if and
only if it is relatively hyperbolic in the sense of Osin.
\end{thm}
\begin{proof}
By Corollary \ref{BCPindep2}, BCP property of $\GH$ follows from
Lemma \ref{BCPOmega}. The hyperbolicity of relative Cayley graph
$\G$ is proven in \cite[Corollary 2.54]{Osin}. Thus, $\GH$ is
relatively hyperbolic in the sense of Farb.

The sufficient part is proven in the appendix of Osin \cite{Osin}
for finitely generated relatively hyperbolic groups. We remark that
the only argument involved to use word metrics with respect to
finite generating sets is in the proof of Lemma 6.12 in \cite{Osin}.
But Osin's argument also works for any proper left invariant metric.
Hence, Osin's proof is through for the countable case.
\end{proof}

\subsection{Relatively quasiconvex subgroups}
In this subsection, we shall explicitly describe a quasi-isometric
map between a relatively quasiconvex subgroup to the ambient
relatively hyperbolic group.

The existence of such a quasi-isometric map is first proven in
\cite[Theorem 10.1]{Hru}, but whose statement or proof does not tell
explicitly how a geodesic is mapped. Using an argument of Short on
the geometry of relative Cayley graph, we carry on a more careful
analysis to construct the quasi-isometric map explicitly.

As a byproduct in the course of the construction, we are able to
produce a new proof of the relative hyperbolicity of relatively
quasiconvex subgroups. This was an open problem in \cite{Osin} and
is firstly answered by Hruska \cite{Hru} using different methods.
During the preparation of this thesis, E. Martinez-Pedroza and D.
Wise \cite{MarWise} gave another elementary and self-contained proof
of this result.

\begin{defn} \cite{Hru}
Suppose $\GH$ is relatively hyperbolic and $d$ is some proper left
invariant metric on $G$. A subgroup $\Gamma$ of $G$ is called
\textit{relatively $\sigma$-quasiconvex} with respect to
$\mathbb{H}$ if there exists a constant $\sigma =\sigma(d) > 0$ such
that the following condition holds. Let $p$ be an arbitrary geodesic
path in $\G$ such that $p_-, p_+ \in \Gamma$. Then for any vertex $v
\in p$, there exists a vertex $w \in \Gamma$ such that $d(u,w) <
\sigma$.
\end{defn}

\begin{cor} \label{QCIndep0} \cite{Hru}
Relative quasiconvexity is independent of the choice of proper left
invariant metrics.
\end{cor}

In fact, when proving relative quasiconvexity, we usually verify the
relative quasiconvexity with respect to some partial distance
function, as indicated in the following corollary.  See an
application of this corollary in the proof of Proposition
\ref{PIISQC}.

\begin{cor} \label{QCIndep}
Suppose $\GH$ is relatively hyperbolic and $\Gamma$ is a subgroup of
$G$. Let $\mathcal A \subset G$ be a finite set and $\da$ the
partial distance function with respect to $\mathcal A$. If there
exists a constant $\sigma = \sigma(\da) > 0$ such that for any
geodesic $p$ with endpoints at $\Gamma$, the vertex set of $p$ lies
in $\sigma$-neighborhood of $\Gamma$ with respect to $\da$. Then
$\Gamma$ is relatively quasiconvex.
\end{cor}

We are going to construct the quasi-isometric map. The relatively
finitely generatedness of $\Gamma$ in Lemma \ref{qcembed} is also
proved by E. Martinez-Pedroza and D. Wise \cite{MarWise}. In
particular, it also follows from a more general result of
Gerasimov-Potyagailo \cite{GePo3}, which states that 2-cocompact
convergence groups are finitely generated relative to a set of
maximal parabolic subgroups.

\begin{lem} \label{qcembed}
Suppose $\GH$ is relatively hyperbolic. Let $\Gamma < G$ be
relatively $\sigma$-quasiconvex. Then $\Gamma$ is finitely generated
by a finite subset $Y \subset G$ with respect to a finite collection
of subgroups
\begin{equation} \label{K}
\mathbb{K} = \{H_i^g \cap \Gamma: |g|_d < \sigma, i \in I, \sharp \;
H_i^g \cap \Gamma = \infty\}.
\end{equation}
Moreover, $X$ can be chosen such that $Y\subset X$ and there is a
$\Gamma$-equivariant quasi-isometric map $\iota: \K \to \G$.
\end{lem}

\begin{proof}
The argument is inspired by the one of \cite[Lemma 4.14]{Osin}.

For any $\gamma \in \Gamma$, we take a geodesic $p$ in $\G$ with
endpoints 1 and $\gamma$. Suppose the length of $p$ is $n$. Let
$g_0=1$, $g_1, \ldots, g_n=\gamma$ be the consecutive vertices of
$p$. By the definition of relative quasiconvexity, for each vertex
$g_i$ of $p$, there exists an element $\gamma_i$ in $\Gamma$ such
that $d(g_i,\gamma_i) < \sigma$.

Denote by $x_i$ the element $\gamma^{-1}_i g_i$, and by $e_{i+1}$
the edge of $p$ going from $g_i$ to $g_{i+1}$. Obviously we have
$\gamma_{i+1} = \gamma_i x_i \lab(e_{i+1}) x^{-1}_{i+1}$.

Set $\kappa = \max\{|x|_d: x \in X\}$. Then $\kappa$ is finite, as
$X$ is finite. Let $Z_0 = \{\gamma \in \Gamma: \norm{\gamma}{d} \leq
2\sigma + \kappa\}$ and $Z_{x,y,i} = \{xhy^{-1} : \; h \in H_i\}
\cap \Gamma$. Since the metric $d$ is proper, the set $B_\sigma :=
\{g \in G: \norm{g}{d} \leq \sigma\}$ is finite.

For simplifying notations, we define  sets $$\Pi = \{(x,y,i):
x,y \in B_\sigma, i \in I \}$$ and
$$\Xi = \{(x,y,i): \sharp \; Z_{x,y,i} = \infty, x,y \in
B_\sigma, i \in I\}.$$

If $e_{i+1}$ is an edge labeled by a letter from $X$, then the
element $x_i \lab(e_{i+1}) x^{-1}_{i+1}$ belongs to $Z_0$.  If
$e_{i+1}$ is an edge labeled by a letter from $H_k$, then $x_i
\lab(e_{i+1}) x^{-1}_{i+1}$ belongs to $Z_{x_i,x_{i+1},k}$. By the
construction, we obtain that the subgroup $\Gamma$ is also generated by
the set
\[
Z := Z_0 \cup \left(\cup_{(x,y,i) \in \Pi} Z_{x, y, i}\right).
\]

For each $(x,y,i) \in \Pi$, if $Z_{x,y,i}$ is nonempty, then we take
an element of the form $x h_i y^{-1} \in Z_{x,y,i}$ for some $h_i
\in H_i$. Denote by $Z_1$ the union of all such elements
$\bigcup_{(x,y,i) \in \Pi} x h_i y^{-1}$. Note that $Z_1 \subset Z$.
Then we have that $\Gamma$ is generated by the set
\[
\hat{Z} := Y \cup \left( \cup_{(z,z,i) \in \Xi} Z_{z, z, i}\right),
\]
where $Y := Z_0 \cup Z_1 \cup \left( \bigcup_{(z, z, i) \in \Pi
\setminus \Xi} Z_{z, z, i}\right)$. Indeed,
for each triple $(x,y,i) \in \Pi$, we have
$$
Z_{x,y,i} = Z_{x,x,i} \cdot xh_i y^{-1}, \; \text{where} \; xh_i
y^{-1} \in Z_1.
$$
On the other direction, it is obvious that $\hat{Z} \subset Z$.

Let $\hat X = X \cup Y \cup B_\sigma$. By the above construction, we
define a $\Gamma$-equivariant map $\phi$ from $\mathscr{G}(\Gamma,
Z)$ to $\mathscr{G}(G, \hat X \cup \mathcal H)$ as follows. For each
vertex $\gamma \in V(\mathscr{G}(\Gamma, Z))$, $\phi(\gamma) =
\gamma$. For each edge $[\gamma,s] \in E(\mathscr{G}(\Gamma, Z))$,
if $s \in Z_0$, then $\phi([\gamma, s])=[\gamma, s]$; if $s \in
Z_{x,y,i}$ for some $(x, y, i) \in \Xi$, then $s = x t y^{-1}$ for
some $t \in H_i$ and we set $\phi([\gamma, s])=[\gamma, x][\gamma x,
t][\gamma x t, y^{-1}]$.

For any  $\gamma_1,\gamma_2 \in V(\mathscr{G}(\Gamma, Z))$,
it is easy to see that $d_{\hat X \cup \mathcal H}(\gamma_1,\gamma_2) <
3d_Z(\gamma_1,\gamma_2)$. For the other direction, we take a
geodesic $q$ in $\G$ with endpoints $\gamma_1,\gamma_2$.

Since $\hat X$ is finite, there exist constants $\lambda \geq 1,c
\geq 0$ depending only on $\hat X$, such that the graph embedding
$\G \hookrightarrow \mathscr{G}(G, \hat X \cup \mathcal H)$ is a
$G$-equivariant ($\lambda,c$)-quasi-isometry. Thus, $q$ is a
($\lambda,c$)-quasigeodesic in $\mathscr{G}(G, \hat X \cup \mathcal
H)$, i.e.
\[
\dxh(\gamma_1,\gamma_2) < \lambda d_{\hat X \cup \mathcal
H}(\gamma_1,\gamma_2) + c.
\]

Since $q$ is a geodesic in $\G$ ending at $\Gamma$, we can apply the
above analysis to $q$ and obtain that $d_Z(\gamma_1,\gamma_2) <
\dxh(\gamma_1,\gamma_2)$. Then we have
\[
d_Z(\gamma_1,\gamma_2) < \lambda d_{\hat X \cup \mathcal
H}(\gamma_1,\gamma_2) +c.
\]
Therefore, $\phi$ is a $\Gamma$-equivariant quasi-isometric map.

We now claim the subgraph embedding $\imath:\mathscr{G}(\Gamma, \hat
Z) \hookrightarrow \mathscr{G}(\Gamma,Z)$ is a $\Gamma$-equivariant
$(2,0)$-quasi-isometry. This is due to the following observation:
every element of $Z$ can be expressed as a word of $\hat{Z}$ of
length at most 2.

Finally, we obtain a $\Gamma$-equivariant quasi-isometric map $\iota
:= \phi \cdot \imath$ from $\K$ to $\mathscr{G}(G,\hat X \cup
\mathcal H)$.
\end{proof}

\begin{rem} \label{nonconj}
Eliminating redundant entries of $\mathbb K$ such that all entries
of $\mathbb K$ are non-conjugate in $\Gamma$, we keep the same
notation $\mathbb K$ for the reduced collection. It is easy to see
the construction of the quasi-isometric map $\iota: \K \to \G$ works
for the reduced $\mathbb K$.
\end{rem}

In the following of this subsection, we assume the
$\Gamma$-equivariant quasi-isometric map $\iota: \K \to \G$ is the
one constructed in Lemma \ref{qcembed}. In particular $X$ is the
suitable chosen relative generating set such that $Y \subset X$.

\begin{lem} \label{qcpair}
Suppose $\GH$ is relatively hyperbolic. Let $\Gamma < G$ be
relatively $\sigma$-quasiconvex. Then the quasi-isometric map
$\iota: \K \to \G$ sends distinct peripheral $\mathbb K$-cosets of
$\Gamma$ to a $d$-distance $\sigma$ from distinct peripheral
$\mathbb H$-cosets of $G$.
\end{lem}
\begin{proof}
Taking into account Lemma \ref{qcembed} and Remark \ref{nonconj}, we
suppose all entries of $\mathbb{K}$ are non-conjugate. We continue
the notations in the proof of Lemma \ref{qcembed}.

By the construction of $\phi$, we can see the map $\phi$ sends the
subset $g Z_{x,x,i}$ to a uniform $d$-distance $\sigma$ from
the peripheral coset $g x H_i$ of $G$ for each $(x,x,i) \in \Xi$ and $g
\in G$. Here $\sigma$ is the quasiconvex constant associated to
$\Gamma$. Observe that $\imath:\mathscr{G}(\Gamma, \hat Z)
\hookrightarrow \mathscr{G}(\Gamma,Z)$ is an embedding. Therefore,
we have the quasi-isometric map $\iota = \phi \cdot \imath$ maps
each peripheral $\mathbb K$-coset to a uniform distance from a
peripheral $\mathbb H$-coset.

We now show the ``injectivity'' of $\iota$ on $\mathbb K$-cosets. Let
$\gamma H_i^g \cap \Gamma$, $\gamma' H_{i'}^{g'} \cap \Gamma$ be
distinct peripheral $\mathbb K$-cosets of $\Gamma$, where $\gamma,
\gamma' \in \Gamma$ and $H_i^g \cap \Gamma, H_{i'}^{g'} \cap \Gamma
\in \mathbb K$.

Using Lemma \ref{periphintersect}, it is easy to deduce that if
$\gamma (H_i^g \cap \Gamma) \gamma^{-1} \cap (H_{i'}^{g'} \cap
\Gamma)$ is infinite, then $i = i'$ and $\gamma \in H_i^g \cap
\Gamma$.

It is seen from the above discussion that there is a uniform
constant $\sigma > 0$, such that $\iota(\gamma H_i^g \cap \Gamma)
\subset N_\sigma(\gamma g H_i)$ and $\iota(\gamma' H_{i'}^{g'} \cap
\Gamma) \subset N_\sigma(\gamma' g' H_{i'})$. It suffices to show
that $\gamma g H_i \neq \gamma' g' H_{i'}$.

Without loss of generality, we assume that $i = i'$. Suppose, to the
contrary, that $\gamma g H_i = \gamma' g' H_i$. Then we have $\gamma
g = \gamma' g' h$ for some $h \in H_i$. It follows that $\gamma g
H_i g^{-1} \gamma^{-1} = \gamma' g' H_i g'^{-1} \gamma'^{-1}$. This
implies that $H_i^{g} \cap \Gamma$ is conjugate to $H_i^{g'} \cap
\Gamma$ in $\Gamma$, i.e. $H_i^{g} \cap \Gamma = (H_i^{g'} \cap
\Gamma)^{\gamma^{-1} \gamma'}$. Since any two entries of $\mathbb
K$ are non-conjugate in $\Gamma$, we have $H_i^{g} \cap \Gamma =
H_i^{g'} \cap \Gamma$. As a consquence, we have $\gamma^{-1} \gamma'
\in H_i^{g} \cap \Gamma$, as $H_i^{g} \cap \Gamma \in \mathbb K$ is
infinite. This is a contradiction, since we assumed $\gamma H_i^g
\cap \Gamma \neq \gamma' H_{i'}^{g'} \cap \Gamma$.

Therefore, $\iota$ sends distinct peripheral $\mathbb K$-cosets of
$\Gamma$ to a uniform distance from distinct peripheral $\mathbb
H$-cosets of $G$.
\end{proof}

Before proceeding to prove the relative hyperbolicity of relatively
quasiconvex subgroups, we need justify the finite collection
$\mathbb K$ in (\ref{K}) as a set of representatives of
$\Gamma$-conjugacy classes of $\mathbb{\hat K}$ in (\ref{Khat}).
\begin{lem} \cite{MarPed2} \label{conjugacy}
Suppose $\GH$ is relatively hyperbolic. Let $\Gamma < G$ be
relatively $\sigma$-quasiconvex. Then the following collection of
subgroups of $\Gamma$
\begin{equation} \label{Khat}
\mathbb{\hat K} = \{H_i^g \cap \Gamma: \sharp \; H_i^g \cap \Gamma =
\infty, g \in G, i \in I\}.
\end{equation}
consists of finitely many $\Gamma$-conjugacy classes. In particular,
$\mathbb K$ is a set of representatives of $\Gamma$-conjugacy
classes of $\mathbb{\hat K}$.
\end{lem}
\begin{proof}
This is proven by adapting an argument of Martinez-Pedroza
\cite[Proposition 1.5]{MarPed2} with our formulation of BCP property
\ref{BCPdef}. We refer the reader to \cite{MarPed2} for the details.
\end{proof}

We are ready to show the relative hyperbolicity of $\KK$. Using
notations in the proof of Lemma \ref{qcembed}, we recall that
$\mathbb K = \{Z_{x,x,i}: (x,x,i) \in \Xi\}$.

\begin{lem} \label{QCISRH}
Suppose $\GH$ is relatively hyperbolic. If $\Gamma < G$ is
relatively $\sigma$-quasiconvex, then $\KK$ is relatively
hyperbolic.
\end{lem}
\begin{proof}
Recall that $\iota$ is the $\Gamma$-equivariant quasi-isometric map
from $\K$ to $\G$. In particular we assumed $Y \subset X$.

We shall prove the relative hyperbolicity of $\Gamma$ using Farb's
definition. First, it is straightforward to verify that $\K$ has the
thin-triangle property, using the quasi-isometric map $\iota$ and
the hyperbolicity of $\G$.

Let $d_G$ be a proper left invariant metric on $G$. Denote by
$d_\Gamma$ the restriction of $d_G$ on $\Gamma$. Obviously
$d_\Gamma$ is a proper left invariant metric on $\Gamma$. We are
going to verify BCP property 1) with respect to $d_\Gamma$, for the
pair $\KK$.  The verification of BCP property 2) is similar.

Let $[\gamma, s]$ be an edge of $\K$, where $s \in Z_{x,x,i}$ for
some $(x,x,i) \in \Xi$. By the construction of $\iota$, $[\gamma,
s]$ is mapped by $\iota$ to the concatenated path $[\gamma,
x][\gamma x, t][\gamma z t, x^{-1}]$, which clearly contains an
$H_i$-component $[\gamma x, t ]$. Note that $\norm{x}{d} \le
\sigma$. To simplify notations, we reindex $\mathbb K = \Kl$.

Given $\lambda \geq 1$ and $c\geq 0$, we consider two $(\lambda,
c)$-quasigeodesics $p,  q$ without backtracking in $\K$ such that
$p_- = q_-$, $p_+ = q_+$. By Lemma \ref{qcpair}, as $p$, $q$ are
assumed to have no backtracking, the paths $\hat p =\iota(p)$, $\hat
q=\iota(q)$ in $\G$ also have no backtracking. Moreover, for each
$H_i$-component $\hat s$ of $\hat p$(resp. $\hat q$), there is a
$K_j$-component $s$ of $p$(resp. $q$) such that $\hat s \subset
\iota(s)$.

Note that paths $\hat p, \hat q$ are ($\lambda'$,$c'$)-quasigeodesic
without backtracking in $\G$ for some $\lambda' \geq 1, c' \geq 1$.
By BCP property of $\GH$, we have the constant $\hat
a=a(\lambda',c',d_G)$. Set $a = \hat a + 2\sigma$, where $\sigma$ is
the quasiconvex constant of $\Gamma$. Let $s$ be a $K_j$-component
of $p$ for some $j \in J$. We claim that if $d_\Gamma(s_-,s_+) > a$,
then there is a $K_j$-component $t$ of $q$ connected to $s$.

By the property of the map $\iota$, there exists an $H_i$-component
$\hat s$ of $\hat p$ such that the following hold $$d_G(\hat s_-,
\iota(s)_-) \le \sigma, \; d_G(\hat s_+, \iota(s)_+) \le \sigma.$$
Thus, we have $d_G(\hat s_-,\hat s_+) > \hat a$. Using BCP property
1) of $\GH$, there exists an $H_i$-component $\hat t$ of $\hat q$,
that is connected to $\hat s$. By the construction of $\iota$, there
is a $K_k$-component $t$ of $q$ for some $k \in J$ such that $\hat t
\subset \iota(t)$.

Since $\hat s$ and $\hat t$ are connected as $H_i$-components,
endpoints of $\hat s$ and $\hat t$ belong to the same $H_i$-coset.
By Lemma \ref{qcpair}, it follows that $k =j$. Furthermore,
endpoints of $s$ and $t$ must belong to the same $K_j$-coset. Hence
$s$ and $t$ are connected in $\K$. Therefore, it is verified that
$\KK$ satisfies BCP property 1).
\end{proof}

\section{Characterization of parabolically embedded subgroups}

\begin{conv}\label{conv0}
Without loss of generality, peripheral structures considered in this
section consist of infinite subgroups. It is easy to see that adding
or eliminating finite subgroups in peripheral structures still gives
relatively hyperbolic groups.
\end{conv}

\subsection{Parabolically embedded subgroups}
Let $\mathbb H = \Hl$ and $\mathbb P = \Pl$ be two peripheral
structures of a countable group $G$. Recall that $\mathbb P$ is an
\textit{extended peripheral structure} for $\GH$, if for each $H_i
\in \mathbb H$, there exists $P_j \in \mathbb P$ such that $H_i
\subset P_j$. Given $P \in \mathbb P$,  we define $\mathbb H_P =
\{H_i: H_i \subset P, i \in I \}$.

\begin{defn}
Suppose $\GH$ is relatively hyperbolic and $\mathbb P$ an extended
peripheral structure for $\GH$. If $\GP$ is relatively hyperbolic,
then $\mathbb P$ is called a \textit{parabolically extended}
structure for $\GH$. Moreover, each $P \in \mathbb P$ is said to be
\textit{parabolically embedded} into $\GH$.
\end{defn}

In this subsection, we assume that $\GH$ is relatively hyperbolic
and $\mathbb P$ is a parabolically extended structure for $\GH$.

Fix a finite relative generating set $X$ for $\GH$ and thus $\GP$.
Since $\GP$ is relatively hyperbolic, by Lemma \ref{Omega}, we
obtain a finite subset $\Omega$ and $\kappa \ge 1$ such that the
inequality (\ref{keyinequality}) holds in $\Gp$.

Due to Lemma \ref{periphintersect} and Convention \ref{conv0}, it is
worth to mention that we have $\mathbb H_P \cap \mathbb H_{P'} =
\emptyset$, if $P, P'$ are distinct in $\mathbb P$. This implies
that each $H \in \mathbb H$ belongs to exactly one $P \in \mathbb
P$.

Since $\mathbb P$ is an extended structure for $\GH$, then for each
$H_i \in \mathbb H$, there exists a unique $P_j \in \mathbb P$ such
that $H_i \hookrightarrow P_j$. By identifying $\mathcal H \subset
\mathcal P$, we regard $\G$ as a subgraph of $\Gp$.

With a slight abuse of notations, a path $p$ in $\G$ will be often
thought of as a path in $\Gp$. The ambiance will be made clear in
the context. The length $\len(p)$ of a path $p$ should also be
understood in the corresponding relative Cayley graphs, but the
values are equal by the natural embedding.

Let $\hat X = X \cup \Omega$. We first show parabolically embedded
subgroups are relatively finitely generated.
\begin{lem}\label{rfg}
Let $\Gamma = P_j \in \mathbb P$ be parabolically embedded into $G$.
Then $Y := \hat X \cap \Gamma$ is a finite relative generating set
for the pair $(\Gamma, \mathbb H_\Gamma)$.
\end{lem}
\begin{proof}
For any $\gamma \in \Gamma$, we take a geodesic $p$ in $\G$ with
endpoints 1 and $\gamma$. In $\Gp$, we can connect $p_-$ and $p_+$
by an edge $e$, labeled by some letter from $\Gamma$, such that
$e_-=p_-$ and $e_+=p_+$. Then the path $c := pe^{-1}$ is a cycle in
$\Gp$. Without loss of generality, we assume $e$ is a
$\Gamma$-component of $c$.

The following two cases are examined separately.

\textit{Case 1.} If there is no $\Gamma$-component of $p$ connected
to $e$ in $\Gp$, then $e$ is an isolated component of $c$. By Lemma
\ref{Omega}, we have
\[
d_{\Omega_j} (e_-,e_+) \le \kappa \len(c) \le \kappa(\len(p) + 1)
\]
where $\Omega_j := \Omega \cap \Gamma$. In particular, there is a
path $q$ in $\Ghat$ labeled by letters from $\Omega_j$, such that
$q_- = e_-$, $q_+ = e_+$ and
$$\len(q) =
d_{\Omega_j}(e_-,e_+).$$ Hence, the element $\gamma$ is a word over
the alphabet $\Omega_j$.

\textit{Case 2.} We suppose that $\{e_1,\ldots, e_i, \ldots,e_n\}$
is the maximal set of $\Gamma$-components of $p$ such that each
$e_i$ is connected to $e$. Then $p$ can be decomposed as
\begin{equation} \label{R1}
p = p_1 e_1 \ldots p_i e_i \ldots p_n e_n p_{n+1}.
\end{equation}
Since $e_i$ is a $\Gamma$-component of $p$, each edge of $e_i$ is
labeled by an element in $\Gamma$. On the other hand, as a subpath
of $p$, $e_i$ has the label $\lab(e_i)$ which is a word over $X \cup
\mathcal H$. Observe that each $H \in \mathbb H$ belongs to exactly
one $P \in \mathbb P$. Thus we obtain that each $\lab(e_i)$ is a
word over $(X \cap \Gamma) \cup \mathbb H_\Gamma$

Since the vertex set $\{e_-, (e_1)_-, (e_1)_+, \ldots, (e_n)_-,
(e_n)_+, e_+\}$ lies in $\Gamma$,  we can connect pairs of
consequent vertices
$$\{e_-,(e_1)_-\}, \dots, \{(e_k)_+, (e_{k+1})_-\}, \dots,
\{(e_n)_+,e_+\}$$ by edges $s_0, \dots,s_k \dots, s_n$ labeled by
letters from $\Gamma$ respectively. We can get $n+1$ cycles $c_k :=
p_k s_k^{-1}, 1 \le k \le n+1$, such that $s_k$ is an isolated
$\Gamma$-component of $c_k$.

As argued in \textit{Case 1} for each cycle $c_k$, we obtain a path
$q_k$ in $\Ghat$ labeled by letters from $\Omega_j$, such that
$(q_k)_- = (e_k)_-$, $(q_k)_+ = (s_k)_+$, $\len(q_k) = d_{\Omega_j}(
(s_k)_-, (s_k)_+)$ and the following inequality holds
\begin{equation} \label{R2}
\len(q_k) \le \kappa \len(c_k) \le \kappa(\len(p_k) + 1).
\end{equation}

In particular, we obtain a path $\hat p$ in $\Ghat$ as follows
\begin{equation} \label{R3}
\hat p := q_1 e_1 \ldots q_i e_i \ldots q_n e_n q_{n+1}
\end{equation}
with same endpoints as $p$. Note that the label $\lab(\hat p)$ is a
word over the alphabet $(\hat X \cap \Gamma) \cup \mathcal
H_\Gamma$. Therefore, $\gamma$ is a word over $(\hat X \cap \Gamma)
\cup \mathcal H_\Gamma$.
\end{proof}

\begin{prop} \label{PIISQC}
Let $\Gamma = P_j \in \mathbb P$ be parabolically embedded into $G$.
Then $\Gamma$ is relatively quasiconvex with respect to $\mathbb H$.
\end{prop}
\begin{proof}
Since $\hat X$ is a finite relative generating set for $\GH$, using
Lemma \ref{Omega}, we obtain a finite set $\Sigma$ and constant $\mu
> 1$ such that the inequality (\ref{keyinequality}) holds in
$\Ghat$.

Let $p$ be a geodesic in $\G$ such that $p_-, p_+ \in \Gamma$. By
Corollary \ref{QCIndep}, it suffices to prove that $p$ lies in a
uniform neighborhood of $\Gamma$ with respect to $d_{\hat X \cup
\Sigma}$.

By Lemma \ref{rfg}, we have a finite relative generating set $Y
\subset \hat X$ for $(\Gamma, \mathbb H_\Gamma)$. Then we have
$\Khat \hookrightarrow \Ghat$. Let $q$ be a geodesic in $\Khat$ such
that $q_- = p_-, q_+=p_+$. We claim that $q$ is a quasigeodesic
without backtracking in $\Ghat$. No backtracking of $q$ is obvious.
We will show the quasigeodesicity of $q$.

We apply the same arguments to $p$, as \textit{Case 2} in the proof
of Lemma \ref{rfg}. Precisely, we decompose $p$ as (\ref{R1}) and
proceed to obtain the inequality (\ref{R2}) and construct a path
$\hat p$ as in (\ref{R3}). Observe that $\hat p$ has the same
endpoints as $p$, and $\lab(\hat p)$ is a word over the alphabet $Y
\cup \mathcal H_\Gamma$. As $\hat p$ can be regarded as path in
$\Khat$, we obtain
\begin{equation} \label{P1}
\len(q) \le \len(\hat p).
\end{equation}

Using the inequality (\ref{R2}), we estimate the length of $\hat p$
as follows
\begin{equation} \label{P2}
\begin{array}{rl}
\len(\hat p) & = \sum_{1 \le k\le n+1} \len(q_k) + \sum_{1 \le k\le
n} \len(e_k) \\
& \le \sum_{1 \le k\le n+1} \kappa\len(p_k) + \sum_{1 \le k\le n}
\len(e_k) + (n+1) \kappa \le 2\kappa \len(p).
\end{array}
\end{equation}

Since $\hat X$ is finite, the embedding $\G \hookrightarrow \Ghat$
is a quasi-isometry. Thus there are constants $\lambda \ge 1, c \ge
0$, such that the geodesic $p$ in $\G$ is a $(\lambda,
c)$-quasigeodesic in $\Ghat$, i.e.:
\begin{equation} \label{P3}
\len(p) < \lambda d_{\hat X \cup \mathcal H}(p_-,p_+)+c.
\end{equation}

Combining (\ref{P1}), (\ref{P2}) and (\ref{P3}), we have
\begin{equation} \label{P4}
\len(q) \le 2\kappa \lambda d_{\hat X \cup \mathcal H}(q_-,q_+) +
2\kappa c
\end{equation}

It is easy to see the above estimates (\ref{P1}), (\ref{P2}) and
(\ref{P3}) can be applied to arbitrary subpath of $q$. Thus the same
inequality as (\ref{P4}) is obtained for arbitrary subpath of $q$.
This proves our claim that $q$ is a $(2\kappa \lambda, 2\kappa
c)$-quasigeodesic without backtracking in $\Ghat$.

As $\kappa \ge 1$, $p$ is a $(2\kappa \lambda, 2\kappa
c)$-quasigeodesic in $\mathscr{G}(G, \hat X \cup \mathcal H)$. Hence
by Lemma \ref{stableqg}, there exists a constant $\epsilon =
\epsilon(\kappa, \lambda, c)$ such that, for each vertex $v \in p$,
there is a phase vertex $u \in q$ such that $d_{\hat X \cup
\Sigma}(u,v) \le \epsilon$.

On the other hand, the vertex set of $q$ lies entirely in $\Gamma$.
Thus $p$ lies in a $\epsilon$-neighborhood of $\Gamma$ with respect
to $d_{\hat X \cup \Sigma}$. Therefore, we have proven the relative
quasiconvexity of $\Gamma$ with respect to $\mathbb H$.
\end{proof}

Lemma \ref{QCISRH} and Proposition \ref{PIISQC} together prove the
following.
\begin{cor} \label{PIISRH}
A parabolically embedded subgroup $\Gamma$ is hyperbolic relative to
$\mathbb H_\Gamma$.
\end{cor}

\subsection{Lifting of quasigeodesics}
In this subsection, we assume that $\GH$ is relatively hyperbolic
and $\mathbb P$ an extended structure for $\GH$. The results
established in this subsection will be applied in subsection 3.3 to
prove Theorem \ref{charpi}.

To make our discussion more transparent, we first note the following
assumption, on which the notion of a lifting path is defined.

\begin{assump} \label{assumptiona}
Each $P_j \in \mathbb P$ is relatively quasiconvex with respect to
$\mathbb H$.
\end{assump}

By Lemma \ref{qcembed}, we assume that each $P_j \in \mathbb P$ is
finitely generated by a finite set $Y_j$ with respect to $\mathbb
H_j:= \mathbb H_{P_j}$. Without loss of generality, we assume $X$ to
be a finite relative generating set for $\GH$ such that $Y_j \subset
X$ for each $j \in J$. So we can identify the relative Cayley graph
$\mathscr{G}(P_j,Y_j \cup \mathcal H_j)$ of $P_j$ as a subgraph of
$\G$.  Thus given any path $p$ of $\Gp$, we can define the
\textit{lifting path} of $p$ in $\G$, by replacing each
$P_j$-component of $p$ by a geodesic segment in $\mathscr{G}(P_j,Y_j
\cup \mathcal H_j)$ with same endpoints.

Precisely, we express the path $p$ in $\Gp$ in the following form
\begin{equation} \label{path}
p = s_0 t_0 \ldots s_k t_k \ldots s_n t_n,
\end{equation}
where $t_k$ are $P_k$-components of $p$ and $s_k$ are labeled by
letters from $X$. It is possible that $P_i = P_j$ for $i \ne j$. We
allow $s_0$ and $t_n$ to be trivial.

Let $\iota_k: \mathscr{G}(P_k,Y_k \cup \mathcal H_k) \hookrightarrow
\G$ be the graph embedding. For each $t_k$, we take a geodesic
segment $\hat t_k$ in $\mathscr{G}(P_k,Y_k \cup \mathcal H_k)$ such
that $(\hat t_k)_- = (t_k)_-$ and $(\hat t_k)_+ = (t_k)_+$. Then the
following constructed path
\begin{equation}\label{liftpath}
\hat p = s_0 \iota(\hat t_0) \ldots s_k \iota(\hat t_k) \ldots s_n
\iota(\hat t_n)
\end{equation}
is the \textit{lifting path} of $p$ in $\G$.

The following two lemmas require only \ref{assumptiona} above.

\begin{lem} \label{nobacktracking}
Lifting of a path without backtracking in $\Gp$ has no backtracking
in $\G$.
\end{lem}
\begin{proof}
We assume the path $p$ and its lifting $\hat p$ decompose as
(\ref{path}) and (\ref{liftpath}) respectively. By way of
contradiction, we assume that, for some $H \in \mathbb H$, there
exist $H$-components $r_1, r_2$ of $\hat p$, such that $r_1, r_2$
are connected. Since $s_k$ are labeled by letters from $X$, we have
$r_1 \subset \iota(\hat t_i), r_2 \subset \iota(\hat t_j)$ for some
$0 \le i, j\le n$.

Note that $\hat t_k$ is a geodesic in $\mathscr{G}(P_k,Y_k \cup
\mathcal H_k)$, which has no backtracking. Hence, by
\ref{assumptiona} and Lemma \ref{qcpair}, $\iota(\hat t_k)$ has no
backtracking in $\G$. It follows that $i \neq j$.

Since $p$ is assumed to have no backtracking, we have that
$\iota(\hat t_i)$ and $\iota(\hat t_j)$ lie entirely in distinct
cosets $g_1 P_i$ and $g_2 P_j$ respectively. On the other hand, as
$r_1$ and $r_2$ are connected $H$-components, their endpoints lie in
the same $H$-coset. By the assumption that $\mathbb H$ consists of
infinite subgroups, each $H \in \mathbb H$ belongs to exact one
subgroup $P \in \mathbb P$. Thus it follows that $g_1 P_i$ and $g_2
P_j$ coincide. This leads to a contradiction with
non-backtrackingness of $p$.
\end{proof}

Similar arguments as above allow one to prove the following.
\begin{lem} \label{fellowtravel}
Suppose $p, q$ are two paths in $\Gp$ such that there are no
connected $P_j$-components of $p, q$ for any $P_j \in \mathbb P$.
Then for any $H_i \in \mathbb H$, their liftings $\hat p, \hat q$
have no connected $H_i$-components.
\end{lem}

To deduce our main result Proposition \ref{liftqc}, we make another
assumption as follows.
\begin{assump}\label{assumptionb}
Let $X$ be a finite relative generate set for $\GH$. There exists
$\kappa \geq 1$ such that for any cycle $o$ in $\Gp$ with a set of
isolated $\Gamma$-components $R = \{r_1, \ldots , r_k\}$, the
following holds
\[
\sum\limits_{r \in R} \dxh(r_-,r_+) \le \kappa \len(o).
\]
\end{assump}
\begin{rem}
Lemma \ref{Omega2} below states that Assumption B will be satisfied
under the assumptions of Theorem \ref{charpi}.
\end{rem}

Taking this assumption into account, we have the following.

\begin{prop} \label{liftqc}
Lifting of a quasigeodesic without backtracking in $\Gp$ is a
quasigeodesic without backtracking in $\G$.
\end{prop}
\begin{proof}
To simplify our proof, we prove the proposition for the lifting of a
geodesic $p$ in $\Gp$. General cases follow from a
quasi-modification of the inequality in (\ref{Q0}) mentioned below.

We assume the path $p$ and its lifting $\hat p$ decompose as
(\ref{path}) and (\ref{liftpath}) respectively. Since $\iota_k$ is
an embedding, we will write $\hat t_k$ instead of $\iota(\hat t_k)$
for simplicity. We shall show the lifting path $\hat p = s_0 \hat
t_0 \ldots s_k \hat t_k \ldots s_n \hat t_n$ is a quasigeodesic in
$\G$.

By Lemma \ref{qcembed}, we have each $\hat t_k$ is a
$(\lambda,c)$-quasigeodesic in $\G$, where the constants $\lambda
\geq 1, c\geq 0$ depend on $P_k$ and $X$. As $\sharp\;\mathbb P$ is
finite, $\lambda$ and $c$ can be made uniform for all $P_k \in
\mathbb P$.

Let $q$ be a geodesic in $\G$ with same endpoints as $\hat p$. Since
$\G \hookrightarrow \Gp$, it is obvious that
\begin{equation} \label{Q0}
\len(p) \le \len(q).
\end{equation}
We consider the cycle $c :=p q^{-1}$ in $\Gp$. For each $t_k$, we
are going to estimate the length of $\hat t_k$ in $\G$.

\textit{Case 1.} The path $t_k$ is isolated in $c$. By
\ref{assumptionb}, there exists a constant $\kappa \ge 1$ such that
\begin{equation} \label{Q1}
\dxh((\hat t_k)_-,(\hat t_k)_+) \leq \kappa \len(c) \leq
\kappa(\len(p)+\len(q)) \leq 2\kappa(\len(q)).
\end{equation}

\textit{Case 2.} The path $t_k$ is not isolated in $c$. Then $t_k$
is connected to some $P_k$-component of $q$, as $p$ is assumed to
have no backtracking in $\Gp$. Let $e_1$ and $e_2$ be the first and
last $P_k$-components of $q$ connected to $t_k$. Note that $e_1$ may
coincide with $e_2$.

Since $e_1$ and $e_2$ are  connected to $t_k$, we can take two edges
$u$ and $v$ labeled by letters from $P_k$ such that
$$u_- = (t_k)_-, u_+ = (e_1)_-, v_- = (t_k)_+, v_+ = (e_2)_+ . $$
Then $c_1 := p_1 u q_1^{-1}$ and $c_2 := v^{-1} p_2 {q_2}^{-1}$ are
two cycles in $\Gp$. By the choice of $P_k$-components $e_1$ and
$e_2$, we deduce that $u$ and $v$ are isolated $P_k$-components of
$c_1$ and $c_2$ respectively. By \ref{assumptionb}, we have the
following inequalities
\begin{equation} \label{Q2}
\dxh(u_-,u_+) \leq \kappa \len(c_1) \leq \kappa(\len(p)+\len(q)+1)
\leq 2\kappa\len(q)+\kappa,
\end{equation}
and
\begin{equation} \label{Q3}
\dxh(v_-,v_+) \leq \kappa \len(c_2) \leq \kappa(\len(p)+\len(q)+1)
\leq 2\kappa\len(q)+\kappa.
\end{equation}
Then it follows from (\ref{Q2}) and (\ref{Q3}) that
\begin{equation} \label{Q4}
\begin{array}{rl}
\dxh((t_k)_-,(t_k)_+) & \leq \len(q) + \dxh(u_-,u_+) + \dxh(v_-,v_+) \\
& \leq (4\kappa+1)\len(q)+2\kappa.
\end{array}
\end{equation}

As $\hat t_k$ can be regarded as a $(\lambda,c)$-quasigeodesic in
$\G$, we estimate the length of $\hat t_k$ in $\G$ by taking into
account (\ref{Q1}) and (\ref{Q4}),
\begin{equation} \label{Q5}
\begin{array}{rl}
\len(\hat t_k) & \leq \lambda \dxh((t_k)_-,(t_k)_+) +c \\
& \leq \lambda (4\kappa+1)  \len(q)+ 2\lambda \kappa + c.
\end{array}
\end{equation}

Finally, we have
\begin{equation} \label{Q6}
\begin{array}{rl}
\len(\hat p) & = \sum_{0 \leq k \leq n} \len(s_i) + \sum_{0 \leq k
\leq n}
\len(\hat t_k) \\
& \leq \len(q) + \len(q) (\lambda (4\kappa+1)  \len(q)+ 2\lambda
\kappa + c) \\
& \leq \lambda (4\kappa+1) (\len(q))^2 + (2 \lambda \kappa + c +
1)\len(q).
\end{array}
\end{equation}

Similarly, we can apply the above estimates to arbitrary subpath of
$\hat p$ to obtain the same quadratic bound on its length as
(\ref{Q6}). It is well-known that in hyperbolic spaces a
sub-exponential path is a quasigeodesic, see Bowditch \cite[Lemma
5.6]{Bow1} for example. Note that $\G$ is hyperbolic. Hence $\hat p$
is a quasigeodesic in $\G$.
\end{proof}

\begin{rem}
In \cite{MarPed}, Martinez-Pedroza proves a specical case of
Proposition \ref{liftqc}, where $\mathbb P$ is obtained from $\mathbb H$ by adding hyperbolically embedded subgroups in the sense of Osin \cite{Osin2}.
\end{rem}

\subsection{Characterization of parabolically embedded subgroups}
Let $\mathbb H = \Hl$ and $\mathbb K \subset \mathbb H$ be two
peripheral structures of a countable group $G$. We will show the
following characterization of parabolically embedded subgroups.

\begin{thm} \label{charpi}
Let $G$ be hyperbolic relative to $\mathbb H$. Assume that

(C0). $\Gamma \subset G$ contains $\mathbb K \subset \mathbb H$,

(C1). $\Gamma$ is relatively quasiconvex,

(C2). $\Gamma$ is weakly malnormal,

(C3). $\Gamma^g \cap H_i$ is finite for any $g\in G$ and $H_i \in
\mathbb H \setminus \mathbb K$. \\
Then $G$ is hyperbolic relative
to $\{\Gamma\} \cup \mathbb H \setminus \mathbb K$.
\end{thm}

Putting in another way, Theorem \ref{charpi} implies the following
\begin{cor}
Under the assumptions of Theorem \ref{charpi}, $\Gamma$ is a
parabolically embedded subgroup of $G$ with respect to $\mathbb K$.
\end{cor}

We now prove Theorem \ref{mainthm} using Theorem \ref{charpi}.
\begin{proof} [Proof of Theorem \ref{mainthm}]
For the sufficient part, Condition (P1) follows from Proposition
\ref{PIISQC}. Since $\GP$ is relatively hyperbolic, Conditions (P2)
and (P3) are direct consequences of Lemma \ref{periphintersect}.

Let $\mathbb P =\{P_1,\ldots, P_j, \ldots, P_n\}$. Recall that
$\mathbb H_{P_j} = \{H_i: H_i \subset P_j; i \in I\}$. Define
peripheral structures $$\mathbb P_k = \{P_1, \ldots, P_k\} \cup
\left(\mathbb H \setminus \cup_{1 \le j \le k} \mathbb
H_{P_j}\right), 0 \le k \le n.$$ Note that $\mathbb P_0 = \mathbb
H$, $\mathbb P_n = \mathbb P$. By definition, we have $\mathbb P_k$
is an extended structure for $(G, \mathbb P_{k-1})$ for each $1 \le
k \le n$. In particular, Conditions (P1)-(P3) imply that $P_k
\subset G$ satisfies Conditions (C0)-(C3) for $(G, \mathbb
P_{k-1})$.  By repeated applications of Theorem \ref{charpi}, we
obtain $\mathbb P_k$ is parabolically extended for $(G, \mathbb
P_{k-1})$. Finally, we prove that $\mathbb P$ is parabolically
extended for $\GH$.
\end{proof}

In what follows, we have all assumptions of Theorem \ref{charpi} are
satisfied.

Choose a finite relative generating set $X$ for $\GH$. Let $\Omega$
the finite set obtained by using Lemma \ref{Omega} for $\G$. To
simplify notations, we denote $\mathbb P = \{\Gamma\} \cup \mathbb H
\setminus \mathbb K$.

Since $\Gamma \subset G$ is assumed to satisfy Conditions (C0)--(C3),
by Lemma \ref{qcembed} we have $\Gamma$ is finitely generated by a
subset $Y$ with respect to $\mathbb K$. Without loss of generality,
we assume $Y \subset X$. So the graph embedding $\iota: \K
\hookrightarrow \G $ is a quasi-isometric map.

Note that $\mathbb P$ satisfies \ref{assumptiona}. So given any path
$p$ of $\G$, we can define the lifting path $\hat p$ in $\Gp$ as in
Subsection 3.2. So we have exactly Lemmas \ref{nobacktracking} and
\ref{fellowtravel}.

Furthermore, by Lemma \ref{Omega2} below, we have \ref{assumptionb}
satisfied in the current setting. So we have the following result by
Proposition \ref{liftqc}.
\begin{prop} \label{liftqc2}
Under the assumptions of Theorem \ref{charpi}. Lifting of a
quasigeodesic without backtracking in $\Gp$ is a quasigeodesic
without backtracking in $\G$.
\end{prop}

The following Lemma \ref{Omega2} is an analogue of Lemma
\ref{Omega}, without assuming that $\GP$ is relatively hyperbolic.
Recall that $\mathbb P = \{\Gamma\} \cup \mathbb H \setminus \mathbb
K$.
\begin{lem}\label{Omega2}
Under the assumptions of Theorem \ref{charpi}. There exists $\mu
\geq 1$ such that for any cycle $o$ in $\Gp$ with a set of isolated
$\Gamma$-components $R = \{r_1, \ldots , r_k\}$, the following holds
\[
\sum\limits_{r \in R} \dxh(r_-,r_+) \le \mu \len(o).
\]
\end{lem}

We defer the proof of Lemma \ref{Omega2} and now finish the proof of
Theorem \ref{charpi} by using Proposition \ref{liftqc2}.
\begin{proof}[Proof of Theorem \ref{charpi}]
We shall prove the relative hyperbolicity of $\GP$ using Farb's
definition.

Let $p q r$ be a geodesic triangle in $\Gp$. We are going to verify
the thinness of $p q r$. Let $\hat p, \hat q, \hat r$ be lifting of
$p, q, r$ in $\G$ respectively. Then by Proposition \ref{liftqc2},
there exists $\lambda \ge 1,c \ge 0$ such that  $\hat p \hat q \hat
r$ is a $(\lambda,c)$-quasigeodesic triangle in $\G$.

Since $\GH$ is relatively hyperbolic, then $\hat p \hat q \hat r$ is
$\nu$-thin for the constant $\nu >0$ depending on $\lambda,c$. That
is to say, the side $\hat p$ belongs to a $\nu$-neighborhood of the
union $q \cup r$. Since $\G \hookrightarrow \Gp$, we have $\dxp(x,y)
\le \dxh(x,y)$ for $x,y \in G$. By the construction of lifting
paths, we have the vertex set of triangle $p q r$ is contained in a
1-neighborhood of the one of triangle $\hat p \hat q \hat r$ in
$\Gp$. Then $p q r$ is $(\nu+1)$-thin in $\Gp$.

Given any $\lambda \ge 1, c\ge 0$, we take two $(\lambda,
c)$-quasigeodesics $p, q$ without backtracking in $\Gp$ with same
endpoints. Let $\hat p, \hat q$ be lifting of $p,q$ in $\G$
respectively. By Proposition \ref{liftqc2}, there exist constants
$\lambda' \ge 1, c' \ge 0$, such that $\hat p, \hat q$ are
$(\lambda', c')$-quasigeodesic without backtracking in $\G$.

Let $\hat X = X \cup \Omega$. Using Lemma \ref{Omega} again, we
obtain a finite set $\Sigma$ and constant $\mu > 1$ such that the
inequality (\ref{keyinequality}) holds in $\Ghat$. Let
$\epsilon=\epsilon(\lambda',c')$ the constant given by Lemma
\ref{stableqg}.

Suppose $s$ is a $\Gamma$-component of $p$ such that no
$\Gamma$-component of $q$ is connected to $s$. To verify BCP
property 1), it suffices to bound $d_{\hat X \cup \Sigma}(s_-,s_+)$
by a uniform constant using Corollary \ref{BCPindep2}. BCP property
2) can be verified in a similar way.

Since endpoints of $s$ belong to the vertex set of $\hat p$, by
Lemma \ref{stableqg}, there exists vertices $\hat u, \hat v \in \hat
q$ such that
$$d_{\hat X}(s_-,\hat u)<\epsilon, \; d_{\hat X}(s_+,\hat
v)<\epsilon.$$ If $\hat u$ is not a vertex of $q$, then $\hat u$
must belong to a lifting of a $\Gamma$-component of $q$. So we can
take a phase vertex $u \in p$ such that $\dxp(u,\hat u) \le 1$.
Otherwise, we set $u =\hat u$. Similarly, we choose a phase vertex
$v$ of $q$ such that $\dxp(v,\hat v) \le 1$. We connect $u,\hat
u$(resp. $v, \hat v$) by a path $e_u$(resp. $e_v$), which consists
of at most one edge labeled by a letter from $\Gamma$. The path
$e_u$ is trivial if $u =\hat u$.

By regarding $p, q$ as paths in $\Gphat$, there exist paths $l$ and
$r$ in $\Gphat$ labeled by letters from $\hat X$, such that $l_- =
s_-, l_+ = \hat u, r_- = s_+, r_+ = \hat v$. Let $$o = s r e_v
[u,v]_q^{-1} e_u^{-1} l^{-1}$$ be a cycle in $\Gphat$, where
$[u,v]_q$ denotes the segment of $q$ between $u$ and $v$. Since
$[u,v]_q$ is a $(\lambda,c)$-quasigeodesic in $\Gphat$, then by the
triangle inequality,
$$\begin{array}{rl}
\len([u,v]_q)  & \le \lambda \dxp(u,v) +c\\
& \le \lambda(\dxp(u,\hat u) + \dxo(\hat u, s_-) + 1 + \dxp(s_+,\hat v) + \dxp(v,\hat v)) +c \\
& \leq \lambda(3 + 2 \epsilon) + c.
\end{array}$$
It follows that $$\begin{array}{rl}
\len(o) & \le \len([u,v]_q) + \dxp(u,\hat u) + \dxo(\hat u, s_-) + 1 + \dxp(s_+,\hat v) + \dxp(v,\hat v)  \\
& \le (\lambda +1) (3 + 2 \epsilon) + c.
\end{array}$$
By Lemma \ref{Omega2}, there exists a constant $\mu \ge 1$ such that
\begin{equation} \label{C1}
\dxh(s_-,s_+) \le \mu \len(o) \le \mu(\lambda +1) (3 + 2 \epsilon) +
c\mu.
\end{equation}

Let $\hat s$ be the lifting of the $\Gamma$-component $s$ in $\G$.
As a subpath of $\hat p$, $\hat s$ is a
$(\lambda',c')$-quasigeodesic in $\G$. Then we have $\len(\hat s)
\le \lambda' \dxh(s_-,s_+) + c'.$

We consider the cycle in $\Ghat$ as follows $$\hat o := \hat s r
[\hat u, \hat u]^{-1}_{\hat q} l^{-1},$$ where $[\hat u,\hat
v]_{\hat q}$ denotes the subpath of $\hat q$ between $\hat u$ and
$\hat v$. As $\hat q$ is a $(\lambda',c')$-quasigeodesic in $\G$, we
have
$$\begin{array}{rl}
\len([\hat u, \hat v]_{\hat q})  & \le \lambda' \dxh(\hat u, \hat v)
+
c' \\
& \le \lambda'(\dxp(\hat u, s_-) + \dxh(s_-,s_+) + \dxp(s_-,\hat v)) + c' \\
& \leq \lambda'(2 \epsilon + \dxh(s_-,s_+)) + c'.
\end{array}$$
It follows that
\begin{equation} \label{C2}
\begin{array}{rl}\len(\hat o) & \le d_{\hat
X}(\hat u, s_-) + \len(s) + d_{\hat X}(s_+,\hat v)  + \len([\hat u,
\hat v]_{\hat q}) \\
& < 2 \epsilon(\lambda'+1)+ \lambda'\dxh(s_-,s_+)+ c'.\end{array}
\end{equation}

It is assumed that no $\Gamma$-component of $q$ is connected to the
$\Gamma$-component $s$ of $p$. By Lemma \ref{fellowtravel}, we
obtain that, for any $H_i \in \mathbb H$, no $H_i$-component of
$\hat s$ is connected to an $H_i$-component of $\hat q$. Moreover,
$\hat s$ has no backtracking by Lemma \ref{qcpair}. Hence every
$H_i$-component of $\hat s$ is isolated in the cycle $\hat o$. Using
Lemma \ref{Omega} for $\Ghat$, we have
$$d_{\hat X \cup \Sigma}(s_-,s_+) < \dxh(s_-,s_+)  \cdot \kappa
\len(\hat o).$$ Observe that $\dxh(s_-,s_+)$ and $\len(\hat o)$ are
upper bounded by uniform constants, as shown in \ref{C1}, \ref{C2}.
Thus the distance $d_{\hat X \cup \Sigma}(s_-,s_+)$ is also
uniformly upper bounded by a constant. Therefore, we have completed
the verification of BCP property 1) for $\GP$
\end{proof}

\subsection{Proof of Lemma \ref{Omega2}}
Under the assumptions of Theorem \ref{charpi}, we now prove Lemma
\ref{Omega2}. Our proof is essentially inspired by Osin's arguments
in \cite{Osin2}. In particular, we need the following two Lemmas
\ref{morseqg} and \ref{Gamma} analogous to Lemmas 3.1 and 3.2 in
\cite{Osin2} respectively.

Let $X$ be a finite relative generating set for $\GH$. Recall that
two paths $p$, $q$ in $\G$ are called \textit{$k$-connected} for $k
\ge 0$, if $$\max\{\dxh(p_-, q_-), \dxh(p_+, q_+)\} \leq k.$$ Since
$\GH$ is relatively hyperbolic, we obtain the finite set $\Omega$ by
using Lemma \ref{Omega} for $\G$.

The following lemma requires only the assumption that $\GH$ is
relatively hyperbolic. It can be proven by combining the proofs of
\cite[Proposition 3.15]{Osin} and \cite[Lemma 3.1]{Osin2} partially.
The details is left to the interested reader, or see the proof in Appendix A of Thesis \cite{Yang}. We also remark that a result of the same spirit as Lemma \ref{morseqg} was obtained in \cite[Proposition 1.11]{Dah2}.
\begin{lem} \label{morseqg}
For any $\lambda \geq 1$, $c \geq 0$, there exists $\alpha_1 =
\alpha_1 (\lambda, c)>0$  such that, for any $k \geq 0$, there
exists $\alpha_2 = \alpha_2 (k, \lambda, c) > 0$ satisfying the
following condition. Let $p$, $q$ be two $k$-connected $(\lambda,
c)$-quasigeodesics in $\G$. If $p$ has no backtracking and $u$ is a
phase vertex on $p$ such that $\min\{\dxh (u, p_-), \dxh (u, p_+ )\}
> \alpha_2$. Then there exists a phase vertex $v$ on $q$ such that
$\dxo(u,v) \leq \alpha_1$.
\end{lem}

Using Lemma \ref{morseqg}, the following lemma, although stated in
geometric terms, is a reminiscent of \cite[Lemma 3.2]{Osin2} and can
be proven along the same line with Conditions (C0)--(C3) on
$\Gamma$. See a proof in Appendix A of Thesis \cite{Yang}.

\begin{lem} \label{Gamma}
For any $\lambda \geq 1$, $c \geq 0$, $k > 0$, there exists $L =
L(\lambda, c, k) > 0$ such that the following holds. Let $p, q$ be
$k$-connected $(\lambda, c)$-quasigeodesics without backtracking in
$\G$ such that $p,q$ are labeled by letters from $\Gamma \setminus
\{1\}$. If $\min\{\len(p), \len(q)\} > L$, then $p$ and $q$ as
$\Gamma$-components are connected in $\Gp$.
\end{lem}

We define a \textit{geodesic $n$-polygon} $P$ in a geodesic metric
space as a collection of $n$ geodesics $p_1, \ldots, p_n$ such that
$(p_i)_+ = (p_{i+1})_-,$ where $i$ is taken modulo $n$. The
following lemma follows from the proof of \cite[Lemma 25]{Ol}, but
is weaker.
\begin{lem} \cite{Ol} \label{polygon}
There are constants $\beta_1 = \beta_1(\delta) > 0, \beta_2 =
\beta_2(\delta) > 0$ such that the following holds for any geodesic
$n$-polygon $P$ in a $\delta$-hyperbolic space. Suppose the set of
all sides of $P$ is divided into three subsets $R$, $S$ and $T$ with
length sums $\Sigma_R$, $\Sigma_S$ and $\Sigma_T$ respectively. If
$\Sigma_R > max\{\beta n, 10^3 \Sigma_S\}$ for some $\beta \geq
\beta_1$. Then there exist distinct sides $p_i \in R, p_j \in R \cup
T$ that contain $\beta_2$-connected segments of length greater than
$10^{-3} \beta$.
\end{lem}

We are now ready to prove Lemma \ref{Omega2}. Its proof uses
crucially the quasi-isometric map $\iota: \K \hookrightarrow \G$.
\begin{proof}[Proof of Lemma \ref{Omega2}]
Without loss of generality, we assume the cycle $o$ in $\Gp$ can be
written as the following form
$$
o = r_1 s_1 \ldots r_n s_n
$$
such that $\{r_1, \ldots, r_n\}$ is the maximal set of
$\Gamma$-components of $o$. Note that $s_n$ is not trivial. We
assume that $\{r_{n_1}, \ldots r_{n_k}\}$ is a set of isolated
$\Gamma$-components of $o$.

Consider the geodesic $2n$-polygon $a_1 b_1 \ldots a_n b_n$, where
$a_i$ and $b_i$ are geodesics in $\G$ with the same endpoints as
$r_i$ and $s_i$ respectively. We divide the sides of the 2n-polygon
into three disjoint sets. Let $R = \{a_{n_1},\ldots, a_{n_k}\}$, $S =
\{b_1, \ldots, b_n\}$ and $T =\{a_i: a_i \notin R\}$. Set $\Sigma_R
=\sum\limits_{i=0}^n \len(a_{n_i})$ and $\Sigma_S
=\sum\limits_{i=1}^k \len(b_{n_i}).$ Obviously $\Sigma_S \le
\len(o)$.

Let $\hat r_i$ be a geodesic segment in $\K$ such that $(\hat r_i)_-
= (r_i)_-$ and $(\hat r_i)_+ = (r_i)_+$.  Since the embedding
$\iota: \K \hookrightarrow \G$ is quasi-isometric, then $\hat r_i$
is a $(\lambda, c)$-quasigeodesic in $\G$ for some $\lambda \ge 1,
c\ge 0$. Moreover, $\hat r_i$ has no backtracking by Lemma
\ref{qcpair}.

Let $\delta$ denote the hyperbolicity constant of $\G$. By the
stability of quasigeodesics in hyperbolic spaces (see \cite{Gro} or
\cite{GH}), there exists a constant $\xi = \xi(\delta,\lambda,c)$
such that $\hat r_i$ have a uniform Hausdorff $\xi$-distance from
$a_i$.

Let $\beta_1 = \beta_1(\delta), \beta_2 = \beta_2(\delta)$ be the
constants provided by Lemma \ref{polygon}, $L=L(\lambda, c,
\beta_2+2\xi)$ the constant provided by Lemma \ref{Gamma}.

It suffices to set $\mu = \max\{ \beta_1, 10^3, (L+2\xi)\cdot
10^3\}$ for showing $\Sigma_R \le \mu \len(o)$. Suppose, to the
contrary, we have $\Sigma_R > \mu  \len(o)$. This yields
\[
\Sigma_R > \mu  \len(o) \ge \max\{ \mu \len(o), 10^3\len(o) \ge
\max\{ \mu \len(o), 10^3 \Sigma_S\}.
\]

By Lemma \ref{polygon}, there are distinct sides $a_j \in R$ and
$a_k \in R \cup T$, having $\beta_2$-connected segments of length at
least $\mu \cdot 10^{-3}$. Therefore, there exist $(\beta_2
+2\xi)$-connected subsegments $q_1 \subset \hat r_j, q_2\subset \hat
r_k$ such that $$\min\{\len(q_1),\len(q_2)\} \ge \mu \cdot
10^{-3}-2\xi \ge L.$$ Since $q_1, q_2$ are $(\lambda
,c)$-quasigeodesics labeled by letters from $\Gamma$, they are
connected by Lemma \ref{Gamma}. Thus, $r_j$ and $r_k$ are connected.
This is a contradiction, since $r_j$ is an isolated
$\Gamma$-component of the cycle $o$.
\end{proof}

\section{Peripheral structures and Floyd boundary}

\subsection{Convergence groups and dynamical quasiconvexity}
Let $M$ be a compact metrizable space. We denote by $\Theta^n M$ the
set of subsets of $M$ of cardinality $n$, equipped with the product
topology.

A \textit{convergence group action} is an action of a group $G$ on
$M$ such that the induced action of $G$ on the space $\Theta^3 M$ is
properly discontinuous. Following Gerasimov \cite{Ge1}, a group
action of $G$ on $M$ is \textit{2-cocompact} if the quotient space
$\Theta^2 M/G$ is compact.

Suppose $G$ has a convergence group action on $M$. Then $M$ is
partitioned into a limit set $\Lambda_M(G)$ and discontinuous domain
$M \setminus \Lambda_M(G)$. The \textit{limit set} $\Lambda_M(H)$ of
a subgroup $H \subset G$ is the set of limit points, where a
\textit{limit point} is an accumulation point of some $H$-orbit in
$M$. An infinite subgroup $P \subset G$ is a \textit{parabolic
subgroup} if the limit set $\Lambda_M(P)$ consists of one point,
which is called a \textit{parabolic point}. The stabilizer of a
parabolic point is always a (maximal) parabolic group. A parabolic
point $p$ with stabilizer $G_p$ := $Stab_G (p)$ is \textit{bounded}
if $G_p$ acts cocompactly on $M \setminus \{p\}$. A point $z \in M$
is a \textit{conical} point if there exists a sequence $\{g_i\}$ in
$G$ and distinct points $a, b \in M$ such that $g_i (z) \to a$ ,
while for all $q \in M \setminus \{z\}$, we have $g_i (q) \to b$.

\begin{conv}
For simplicity, we often denote by $G \curvearrowright M$ a
convergence group action of $G$ on a compact metrizable $M$.
\end{conv}

Let us begin with the following simple observation.
\begin{lem} \label{limitsetmap}
Suppose a group $G$ admits convergence group actions on compact
spaces $M$ and $N$ respectively. If there is a $G$-equivariant
surjective map $\phi$ from $M$ to $N$, then for any $H<G$, $\phi
(\Lambda_M(H)) = \Lambda_N(H)$.
\end{lem}

\begin{proof}
Given $x \in \Lambda_M(H)$, by definition, there exists $z \in M$
and $\{h_n\} \subset H$ such that $h_n(z) \to x$ as $n \to \infty$.
Since $\phi$ is a $G$-equivariant map, we have $h_n(\phi(z)) =
\phi(h_n z) \to \phi(x)$. Then $\phi(x)$ is the limit point of of
sequence $\{h_n(\phi(z))\}$ and thus $\phi(x) \in \Lambda_N(H)$.

Conversely, for any $y \in \Lambda_N(H)$, there exists $z \in N$ and
$\{h_n\} \subset H$ such that $h_n(z) \to y$. Take $w \in M$ such
that $\phi(w)=z$. We have $\phi(h_n w) =h_n \phi(w) = h_n z \to y$.
After passage to a subsequence, we assume  $x \in \Lambda_M(H)$ to
be the limit point of sequence $\{h_n w\}$. Then $\phi( h_n w) \to
\phi(x)$ by the continuity of $\phi$. It follows that $\phi(x)=y$.
Therefore, we obtain $\Lambda_N(H) \subset \phi (\Lambda_M(H))$.
\end{proof}

\begin{rem}
Note that in general $\phi^{-1} (\Lambda_N(H)) = \Lambda_M(H)$ is
not true. This is readily seen from Lemma \ref{kernel} below.
\end{rem}

\begin{defn}
A subgroup $H$ of a convergence group action $G \curvearrowright M$
is \textit{dynamically quasiconvex} if the following set
\[
\{gH \in G/H : g \Lambda_N(H) \cap K \neq \emptyset, g \Lambda_N(H)
\cap L \neq \emptyset\}
\]
is finite, whenever $K$ and $L$ are disjoint closed subsets of $M$.
\end{defn}
\begin{rem}
The notion of dynamical quasiconvexity was introduced by Bowditch
\cite{Bow2} in hyperbolic groups and is proven there to be
equivalent to the geometrical quasiconvexity.
\end{rem}

In the following lemma, we show that dynamical quasiconvexity is
kept under an equivariant quotient.
\begin{lem} \label{quotient}
Suppose a group $G$ admits convergence group actions on compact
spaces $M$ and $N$ respectively. Assume, in addition, that there is
a $G$-equivariant surjective map $\phi$ from $M$ to $N$. If $H
\subset G$ is dynamically quasiconvex with respect to $G
\curvearrowright M$, then it is dynamically quasiconvex with respect
to $G \curvearrowright N$.
\end{lem}
\begin{proof}
Given any disjoint closed subsets $K, L$ of $N$, we are going to
bound the cardinality of the following set
\[
\Theta = \{gH \in G/H : g \Lambda_N(H) \cap K \neq \emptyset, g
\Lambda_N(H) \cap L \neq \emptyset\}.
\]

Let $K'=\phi^{-1}(K)$ and $L'=\phi^{-1}(L)$. Obviously $K' \cap L' =
\emptyset$. For each $gH \in \Theta$, we claim $g \Lambda_M(H) \cap
K' \neq \emptyset$. Otherwise, we have then
\[
\phi(g\Lambda_M(H)) \cap \phi(K') = g \phi(\Lambda_M(H)) \cap K =
\emptyset.
\]
By Lemma \ref{limitsetmap}, we have $g\Lambda_N(H) \cap K =
\emptyset$. This is a contradiction. Hence $g \Lambda_M(H) \cap K'
\neq \emptyset$. Similarly, we have $g\Lambda_M(H) \cap L' \neq
\emptyset$.

By the dynamical quasiconvexity of $H$ with respect to $G
\curvearrowright M$, we have $\Theta$ is a finite set. Thus, $H$ is
dynamically quasiconvex with respect to $G \curvearrowright N$.
\end{proof}

\begin{defn}
A convergence group action of $G$ on $M$ is \textit{geometrically
finite} if every limit point of $G$ in $M$ is either a conical or
bounded parabolic.
\end{defn}

We now summarize as follows the equivalence of several dynamical
formulations of relative hyperbolicity. Theorems \ref{GFRH} and
\ref{GFQC} shall enable us to translate the results established in
previous sections in dynamical terms.

\begin{thm}\label{GFRH} \cite{Bow1}\cite{Ge1}\cite{Tukia}\cite{Yaman}
Suppose a finitely generated group $G$ acts on $M$ as a convergence
group action. Let $\mathbb{P}$ be a set of representatives of the
conjugacy classes of maximal parabolic subgroups.
Then the following statements are equivalent: \\
(1) The pair $\GP$ is relatively hyperbolic in the sense of Farb,\\
(2) $G \curvearrowright M$ is geometrically finite,\\
(3) $G \curvearrowright M$ is a 2-cocompact convergence group
action.
\end{thm}
\begin{rem}
The direction $(1) \Rightarrow (2)$ is due to Bowditch \cite{Bow1};
$(2) \Rightarrow (1)$ is proved by Yaman \cite{Yaman}; $(2)
\Rightarrow (3)$ is implied in the work of Tukia \cite[Theorem 1
C]{Tukia}; $(3) \Rightarrow (2)$ is proven in Gerasimov \cite{Ge1}
without assuming that $G$ is countable and $M$ metrizable.
\end{rem}

In Theorem \ref{GFRH}, the limit set of $G$ with respect to $G \curvearrowright M$
will be referred as \textit{Bowditch boundary} of the relatively hyperbolic group $G$. We
shall often write it as $\TGP$, with reference to a particular
peripheral structure $\mathbb P$. It is shown in \cite{Bow1} that
Bowditch boundary is well-defined up to a $G$-equivariant
homeomorphism.

In different contexts, we can formulate the corresponding notions of
relative quasiconvexity, which are proven to be equivalent.
\begin{thm} \label{GFQC}\cite{GePo1} \cite{GePo4} \cite{Hru}
Suppose a finitely generated group $G$ acts geometrically finitely
on $M$. Let $\Gamma$ be a subgroup of $G$.
Then the following statements are equivalent: \\
(1) $\Gamma$ is relatively quasiconvex,\\
(2) $\Gamma \curvearrowright \Lambda_M(\Gamma)$ is geometrically
finite,\\
(3) $\Gamma$ is dynamical quasiconvex with respect to $G
\curvearrowright M$.
\end{thm}
\begin{rem}
The equivalence $(1) \Leftrightarrow (2)$ is proved by Hruska
\cite{Hru} for countable relatively hyperbolic groups; $(1)
\Leftrightarrow (3)$ is proven in Gerasimov-Potyagailo \cite{GePo1}.
\end{rem}

Lastly we recall a useful result about peripheral subgroups of
finitely generated relatively hyperbolic groups.
\begin{lem} \label{undistorted}\cite{DruSapir}\cite{Osin}\cite{Ge1}\cite{Hru}
Suppose $G$ is finitely generated and hyperbolic relative to
$\mathbb H$. Then each $H \in \mathbb H$ is undistorted in $G$.
Moreover $H$ is relatively quasiconvex in any relatively hyperbolic
$\GP$.
\end{lem}
\begin{rem}
The undistortedness of peripheral subgroups are proved by Osin
\cite{Osin}, Drutu-Sapir \cite{DruSapir} and Gerasimov \cite{Ge1},
using quite different methods. The last statement is proved by
Hruska \cite{Hru}.
\end{rem}

\subsection{Floyd boundary and relative hyperbolicity}
In this subsection, we first briefly recall the work of Gerasimov
\cite{Ge2} and Gerasimov-Potyagailo \cite{GePo1} on Floyd maps.
Based on their results, the Bowditch boundary with respect to a
parabolically extended structure is shown as an equivalent quotient,
and then the kernel of such an equivariant map is described.

From now on, unless explicitly stated, $G$ is always assumed to be
finitely generated by a fixed finite generating set $X$.

In \cite{Floyd}, Floyd introduced a compact boundary for a finitely
generated group $G$. Let $f$ be a suitable chosen function
satisfying Conditions (3)--(4) in \cite{GePo2}. We first rescale the
length of each edge $e$ of $\Gx$ by $f(n)$, where $n$ is the word
distance of the edge $e$ to $1 \in G$. Then we take length metric on
$\Gx$ and get the Cauchy completion $\Gf$ of $\Gx$. The complete
metric $\rho$ on $\Gf$ is called \textit{Floyd metric}. The
completion $\Gf$ is compact, and the remainder $\Gf \setminus G$ is
defined to be the \textit{Floyd boundary} $\pGf$ of $G$ with respect
to $f$.

If $\pGf$ consists of 0, 1 or 2 points then it is said to be
\textit{trivial}. Otherwise, it is uncountable and is called
\textit{nontrivial}. If $\pGf$ is nontrivial, then $G$ acts on
$\pGf$ as a convergence group action, by a result of Karlsson
\cite{Ka}.

The following Floyd map theorem due to Gerasimov \cite{Ge2} is key
to our study of peripheral structures.

\begin{thm}\label{floydmap}\cite{Ge2}
Suppose $G \curvearrowright M$ is 2-cocompact and $M$ contains at
least 3 points. Then there exists a continuous $G$-equivariant map
$\phi: \pGf \to M$, where $f(n) = \alpha^n$ for some $\alpha \in ]0,
1[$ sufficiently close to 1. Furthermore $\Lambda(G) = \phi (\pGf)$.
\end{thm}

The map $\phi$ given by Theorem \ref{floydmap} is called
\textit{Floyd map}. According to the discussion in \cite{GePo2}, the
Floyd map $\phi$ defines a closed $G$-invariant equivalent relation
$\omega :=\{(x,y): \phi(x) = \phi(y), x,y \in \pGf \}$, which
induces a \textit{shortcut pseudometric} $\tilde \rho$ on $\pGf$.
This shortcut pseudometric is characterized as the maximal
pseudometric, among which vanishes on $\omega$ and is less then the
Floyd metric $\rho$. See \cite{GePo2} for more details.

Recall that $\TGH$ denotes the Bowditch boundary of $G
\curvearrowright M$, where $\mathbb H$ is a set of representatives
of the conjugacy classes of maximal parabolic subgroups of $G
\curvearrowright M$.

Moreover, the push-forward of $\tilde \rho$ by $\phi$ is shown to be
a metric on $\TGH$ in \cite{Ge2}, which is called \textit{shortcut
metric} (still denoted by $\tilde \rho$). Thus, $\phi$ is a distance
decreasing map from $(\pGf, \rho)$ to $(\TGH, \tilde\rho)$:
\begin{equation} \label{floydecrease}
\forall x, y \in \pGf : \rho(x, y) \ge \tilde\rho (\phi (x), \phi
(y)).
\end{equation}

\begin{conv}
Given a subgroup $J \subset G$, we denote by $\LF(J)$ and $\LH(J)$
limit sets with respect to $G \curvearrowright \pGf$ and $G
\curvearrowright \TGH$ respectively.
\end{conv}

We now recall the characterization of the ``kernel'' of Floyd maps
given in \cite{GePo2}. Note that a more complete characterization
appears in \cite{GePo3}, but here we do not need that deeper result.

\begin{thm}\label{floydker} \cite{GePo2}
Suppose $G \curvearrowright T$ is 2-cocompact. Let $\phi : \pGf \to
T$ be a $G$-equivariant map. Then
\[
\phi^{-1} (p) = \LF(G_p)
\]
for any parabolic point $p \in T$. Moreover, the multivalued inverse
map $\varphi^{-1}$ is injective on conical points of $G
\curvearrowright T$.
\end{thm}

In the following two lemmas, we shall show the Bowditch boundary
with respect to an extended parabolically structure can be described
in a nice way.

\begin{lem} \label{equivarant}
Suppose $\GH$ is relatively hyperbolic. Let $\mathbb P$ be a
parabolically extended structure for $\GH$.  Then there exists a
$G$-equivariant surjective map $\varphi$ such that the following
diagram commutes
\begin{equation}
\xymatrix{
  \partial_f G \ar[dr]_{\phi_2} \ar[r]^{\phi_1}
                & \TGH \ar[d]^{\varphi}  \\
                & \TGP            }
\end{equation}
where $\phi_1$ and $\phi_2$ are Floyd maps given by Theorem
\ref{floydmap}. Furthermore, $\varphi$ is a distance decreasing map
with respect to the shortcut metrics $d_{\mathbb H}$ and $d_{\mathbb
P}$.
\end{lem}
\begin{proof}
The following function $\varphi$ is well-defined:
\[
\forall x \in \TGH: \varphi(x) = \phi_2 \phi_1^{-1}(x).
\]
It is easy to verify that $\varphi$ is a $G$-equivariant continuous map.





We now show the last statement of this lemma. Let $\omega_1$ and
$\omega_2$ be $G$-invariant equivalence relations induced by Floyd
map $\phi_1$ and $\phi_2$ respectively. Observe that $\omega_1
\subset \omega_2$. Thus it follows easily that
\[
\forall x, y\in \TGH: d_{\mathbb H}(x,y) \geq d_{\mathbb
P}(\varphi(x),\varphi(y)).
\]
from the definition of shortcut pseudometrics on $\Gf$.
\end{proof}

The following lemma follows easily from Theorem \ref{floydker} and
describes the kernel of the map $\varphi$ defined in Lemma
\ref{equivarant}.

\begin{lem} \label{kernel}
Suppose $\GH$ is relatively hyperbolic and $\mathbb P$ is a
parabolically extended structure for $\GH$.  Let $\varphi : \TGH \to
\TGP$ be the $G$-equivariant surjective map provided by Lemma
\ref{equivarant}. Then
\[
\varphi^{-1} (p) = \LH(G_p)
\]
for any parabolic point $p \in \TGP$. Moreover, the multivalued
inverse map $\varphi^{-1}$ is injective on conical points of $G
\curvearrowright \TGP$.
\end{lem}
\begin{proof}
Observe that $\varphi^{-1}(p) = \phi_1 \phi_2^{-1}(p)$ for any $p
\in \TGP$. Suppose $\varphi(a) = \varphi (b) = p$ for $a, b \in
\TGH$, i.e. $a, b \in \varphi^{-1}(p)$. If $p$ is conical with
respect to $G \curvearrowright \TGP$, then $\phi_2^{-1}(p)$ consists
of one single point. Thus $a=b$.

If $p$ is bounded parabolic with respect to $G \curvearrowright
\TGP$, then $\phi_2^{-1} (p) = \LF(G_p)$ using Theorem \ref{floydker}.
By Lemma \ref{limitsetmap}, we obtain $\varphi^{-1} (p) =
\phi_1(\LF(G_p)) = \LH(G_p)$. The proof is complete.
\end{proof}

\subsection{Proof of Theorem \ref{mainthm2}}

The proof of Theorem \ref{mainthm2} is divided into the following
two propositions. Taking into account Theorem \ref{GFQC}, the first
proposition follows immediately from Lemmas \ref{limitsetmap} and
\ref{equivarant}.

\begin{prop} \label{dynconv1}
Suppose $\GH$ is relatively hyperbolic and $\mathbb P$ is a
parabolically extended structure for $\GH$. If $\Gamma \subset G$ is
relatively quasiconvex in $G$ with respect to $\mathbb H$, then
$\Gamma$ is relatively quasiconvex in $G$ with respect to $\mathbb
P$.
\end{prop}

By Theorem \ref{GFQC}, the second statement of Theorem
\ref{mainthm2} is restated in the following dynamical terms.

\begin{prop} \label{dynconv2}
Suppose $\GH$ is relatively hyperbolic and $\mathbb P$ is a
parabolically extended structure for $\GH$. Let $\Gamma \subset G$
acts geometrically finitely on $\LP(\Gamma)$. Then $\Gamma$ acts
geometrically finitely on $\LH(\Gamma)$ if and only if $\Gamma \cap
P_j^g$ acts geometrically finitely on $\LH(\Gamma \cap P_j^g)$ for
any $j \in J$ and $g \in G$.
\end{prop}
\begin{proof}
$\Rightarrow$: By Lemma \ref{PIISQC}, each $P_j$ is relatively
quasiconvex with respect to $\mathbb H$. It is a well-known fact
that the intersection of two relatively quasiconvex subgroups is
relatively quasiconvex, see for example, \cite{Hru} and
\cite{MarPed2}. Hence we have $\Gamma \cap P_j^g$ is relatively
quasiconvex with respect to $\mathbb H$, and then acts geometrically
finitely on its limit set $\LH(\Gamma \cap P_j^g)$.

$\Leftarrow$: By Lemma \ref{equivarant}, the map $\varphi: \TGH \to
\TGP$ is a distance decreasing function with respect to the induced
shortcut metrics $d_{\mathbb H}$ and $d_{\mathbb P}$.

Since $\Gamma$ is relatively quasiconvex with respect to $\mathbb
P$, then the following set
\[
\{\Gamma \cap P_j^g: \sharp \; \Gamma \cap P_j^g = \infty , g\in G,
j \in J\}
\]
contains finitely many $\Gamma$-conjugacy classes, say $\{Q_1,
\ldots,Q_n\}$. By Theorem \ref{GFRH}, each $Q_i$ acts
2-cocompactly on $\LH(Q_i)$.  We shall show that $\Gamma$ also acts
2-cocompactly on $\LH(\Gamma)$.

Since $\Gamma$ acts 2-cocompactly on $\LP(\Gamma)$, these exists
$\epsilon_0 > 0$ such that for any $(x,y)\in \Theta^2(\LP(\Gamma))$,
there exists $\gamma \in \Gamma$ satisfying $d_{\mathbb P}(\gamma
x,\gamma y)>\epsilon_0$.  Similarly, we have  a positive constant
$\epsilon_i > 0$ for each $i \in I$ such that for any $(x,y)\in
\Theta^2(\LH(Q_i))$, there exists $\gamma \in Q_i$ satisfying
$d_{\mathbb H}(\gamma x,\gamma y)>\epsilon_i$.

Let $\epsilon := \min\{\epsilon_0,\min\{\epsilon_i: i \in I\}\}$. We
now define a compact $L \subset \Theta^2(\LH(\Gamma))$ as follows
\[
L=\{(x, y)\in \Theta^2(\LH(\Gamma)): d_{\mathbb H}(x,y) \geq
\epsilon \}.
\]
Then we claim $L$ is a fundamental domain of $\Gamma$ on
$\Theta^2(\LH(\Gamma))$.

Given distinct points $p,q \in \LH(\Gamma)$, we have the following
two cases to consider:

\textit{Case 1.} $\phi(p) \neq \phi(q)$. Then there exists $\gamma_0 \in
\Gamma$ such that
\[
d_{\mathbb P}(\gamma_0 (\varphi(p)),\gamma_0(\varphi(q))) =
d_{\mathbb P}(\varphi(\gamma_0 p),\varphi(\gamma_0 q)) >\epsilon_0 >
\epsilon.
\]

Since $\varphi$ is a distance
decreasing map, we have $d_{\mathbb H}(\gamma_0 p,\gamma_0 q) \geq
d_{\mathbb P}(\phi(\gamma_0 p),\phi(\gamma_0 q))$. This implies
$\gamma_0(p,q)\in L$.

\textit{Case 2.} $\phi(p) = \phi(q)$. By Lemma \ref{kernel}, we have
the points $p,q$ lie in the limit set $\LH(Q_i^\gamma)$ for some $1
\leq i \leq n, \gamma \in \Gamma$, i.e.
$(\gamma^{-1}(p),\gamma^{-1}(q)) \in \LH(Q_i)$. Then there exists an
element $\gamma_i$ from $Q_i$ such that $d_{\mathbb H}(\gamma_i
\gamma^{-1}(p), \gamma_i \gamma^{-1}(q)) > \epsilon_i > \epsilon$.
This implies that $\gamma_i \gamma^{-1}(p,q)\in L$.

Combining the above two cases, we showed that $\Gamma$ acts
2-cocompactly ad thus geometrically finitely on $\LH(\Gamma)$.
\end{proof}

\begin{rem}
Using an argument of \cite{MarPed}
with Proposition \ref{liftqc}, one is able to obtain the full
generality of Theorem \ref{mainthm2} for countable relatively
hyperbolic groups. We leave the details to the interested reader.
\end{rem}

The proof of Proposition \ref{dynconv2} also produces the following
result.
\begin{thm} \label{GFFloyd}(Theorem \ref{mainthm4})
Suppose $\GH$ is relatively hyperbolic. Then $G$ acts geometrically
finitely on $\pGf$ if and only if each $H \in \mathbb H$ acts
geometrically finitely on $\LF(H)$.
\end{thm}
\begin{proof}
$\Rightarrow$: Note that each $H \in \mathbb H$ is undistorted in
$G$ by Lemma \ref{undistorted}. Since $G$ acts geometrically
finitely on $\pGf$, then by Theorem \ref{GFQC} each $H \in \mathbb
H$ acts geometrically finitely on $\LF(H)$.

$\Leftarrow$: In particular, we use the Floyd map $\phi: \pGf \to
\TGH$ instead of the map $\varphi$ in the proof of Proposition
\ref{dynconv2}. Note that $F$ is also a distance decreasing map with
respect to $\rho$ and $d_{\mathbb H}$. The other arguments are
exactly the same as Proposition \ref{dynconv2}.
\end{proof}

\subsection{Some applications}
In this subsection, we give some preliminary results on general
peripheral structures. The first result roughly states that if a
finitely generated group acts geometrically finitely on its Floyd
boundary, then every peripheral structure to which it may be
hyperbolic relative are parabolically extended for a canonical
peripheral structure. This is a direct corollary to Theorem
\ref{floydmap}.

\begin{cor} \label{structure} Suppose $G$ acts
geometrically finitely on $\pGf$ and $\GP$ is relatively hyperbolic.
Then $\mathbb P$ is parabolically extended for $\GH$, where $\mathbb
H$ comprises a suitable choice of representatives of the conjugacy
classes of maximal parabolic subgroups with respect to $G
\curvearrowright \pGf$, and possibly a trivial subgroup.
\end{cor}

\begin{proof}
Let $\phi: \pGf \to \TGP$ be the Floyd map given by Theorem
\ref{floydmap}. Let $\tilde{\mathbb H}$ be the collection of maximal
parabolic subgroups with respect to $G \curvearrowright \pGf$.

\begin{claim1}
For each $H \in \tilde{\mathbb H}$, there exists $g \in G$ and $j
\in J$ such that $H \subset P_j^g$.
\end{claim1}
\begin{proof}[Proof of Claim 1]
As $\LF(H)$ is a parabolic point, then $\phi(\LF(H))$ is also fixed
by $H$.  Hence $\LP(H)$ consists of one point or two points. If
$\LP(H)$ is one point, then $H$ contains no hyperbolic elements. By
\cite[Theorem 3A]{Tukia}, the stabilizer of $\LP(H)$ is a maximal
parabolic subgroup for the action $G \curvearrowright \TGP$. So the
claim is proved in this case.

We now show that $\LP(H)$ could not consist of two points. Suppose
not. Let $q$ be the other point in $\LP(H)$. Then the preimage
$\phi^{-1}(q)$ is $H$-invariant. Take a point $z \in \phi^{-1}(q)$.
As $H$ acts properly discontinuously on $\pGf \setminus \{\LF(H)\}$,
then the orbit $H(z)$ should converge to $\LF(H)$. However, we have
$\phi(H(z))$ and $\phi(\LP(H))$ are distinct points.  This
contradicts to the continuity of $\phi$.
\end{proof}

Let $\mathbb H_j$ be a set of representatives of the conjugacy
classes of maximal parabolic subgroups with respect to $P_j
\curvearrowright \LF(P_j)$.

\begin{claim2}
The union $\mathbb H := \cup_{j \in J} \mathbb H_j$ is a set of
representatives of $\tilde{\mathbb H}$.
\end{claim2}

\begin{proof}[Proof of Claim 2]
By Claim 1,  we have that $\mathbb H$ contains at least a set of
representatives of the conjugacy classes of $\tilde{\mathbb H}$.
Moreover, it is easy to verify that no two entries of $\mathbb H$ is
conjugate in $G$.  The claim is thus proved.
\end{proof}

If there exists a parabolic subgroup $P \in \mathbb P$ such that $P$
is a hyperbolic group, then we may add the trivial subgroup into
$\mathbb H$. Then by the choice of $\mathbb H$, we have that $P$ is
parabolically extended for $\GH$.
\end{proof}

In view of Corollary \ref{structure}, Theorem \ref{mainthm2} gives
the following corollary, concerning about ``universal'' relatively
quasiconvex subgroups in certain classes of relatively hyperbolic
groups.

\begin{cor}
If $G$ acts geometrically finitely on $\pGf$ and $\GP$ is relatively
hyperbolic. Then relatively quasiconvex subgroups of $G$ with
respect to $G \curvearrowright \pGf$ are relatively quasiconvex with
respect to $\GP$.
\end{cor}

Moverover, the following conjecture is made by Olshanskii-Osin-Sapir on the
relationship between relatively hyperbolic groups and their Floyd boundaries.
\begin{conjA}\label{conjA}\cite{OOS}
If a finitely generated group has non-trivial Floyd boundary, then
it is hyperbolic relative to a collection of proper subgroups.
\end{conjA}

In \cite{BDM}, Behrstock-Drutu-Mosher studied Dunwoody's
inaccessible group $J$ which is constructed in \cite{Dun}. In
particular, they proved that there exists no collection $\mathbb P$
of NRH proper subgroups such that $J$ is hyperbolic relative to
$\mathbb P$. Moreover, we have the following observation.

\begin{prop} \label{Dunwoody}
Dunwoody's group $J$ in \cite{Dun} does not act geometrically
finitely on its Floyd boundary.
\end{prop}

\begin{proof}
By way of contradiction, we suppose $J
\curvearrowright
\partial_f J$ is geometrically finite. Let $\mathbb P$ be a set of
representatives of the conjugacy classes of maximal parabolic
subgroups with respect to $J \curvearrowright \partial_f J$. Then
the Floyd boundary $\partial_f J$ is same as the Bowditch boundary
$\TGP$. Moreover, the limit set $\LF(P)$ of each $P \in \mathbb P$
consists of only one point.

By Proposition 6.3 in Behrstock-Drutu-Mosher \cite{BDM}, there
exists a subgroup $\Gamma \in \mathbb P$ such that $\Gamma$ is
hyperbolic relative to a collection of proper subgroups $\mathbb K =
\Kl$. By Corollary 1.14 in \cite{DruSapir}, we have that $J$ is
hyperbolic relative to $\mathbb H := \mathbb K \cup (\mathbb P
\setminus \{\Gamma\})$.

By Theorem \ref{floydmap}, we have a $G$-equivalent Floyd map
$\varphi: \TGP = \partial_f J \to \TGH$. Note that
$\Lambda_f(\Gamma)$ consists of one point. Using Theorem
\ref{floydker}, we will obtain that $\varphi$ maps
$\Lambda_f(\Gamma)$ to different points $\LH(K_j)$. This gives a
contradiction. Hence, the action $J \curvearrowright \partial_f J$
is not geometrically finite
\end{proof}

\begin{rem}
Note that a more direct proof (without using \cite[Corollary 1.14
]{DruSapir}) follows from \cite[Theorem C]{GePo3} and Theorem
\ref{floydmap}. A consequence of \cite[Theorem C]{GePo3} says that
if $\GH$ is relatively hyperbolic, then there is a particular Floyd
function $f$ such that for each $H \in \mathbb H$, the limit set
$\LF(H)$ is homeomorphic to its Floyd boundary $\partial_f H$. On
the other hand, Theorem \ref{floydmap} implies that a non-elementary
relatively hyperbolic group has a non-trivial Floyd boundary. So if
$J$ acts geometrically finitely on its Floyd boundary, then Floyd
boundary of every $P \in \mathbb P$ consists of one point. As above,
by \cite[Proposition 6.3]{BDM}, there exists $\Gamma \in \mathbb P$
acting non-trivially on a compactum, which contradicts to Theorem
\ref{floydmap}.
\end{rem}

Recall that a group $H$ is said \textit{Non-Relatively Hyperbolic}
(NRH) if $H$ is not hyperbolic relative to any collection of proper
subgroups. As suggested by Theorem \ref{Dunwoody}, it seems reasonable to
conjecture the following.
\begin{conjB}\label{conjB}
If a finitely generated group is hyperbolic relative to a collection
of NRH proper subgroups, then it acts geometrically finitely on its
Floyd boundary.
\end{conjB}

As a matter of fact, the converse of Conjecture B is true by
Corollary \ref{structure}.

Although Conjectures A and B appear to be different claims, they
turn out to be equivalent by the following simple arguments.
\begin{prop}\label{conjecture}
Conjecture A is equivalent to Conjecture B.
\end{prop}
\begin{proof}
\textbf{Conjecture A implies Conjecture B}: Suppose Conjecture B is
false. Then there exists a relatively hyperbolic group $G$ with
respect to a collection $\mathbb H$ of NRH proper subgroups such
that $G$ does not act geometrically finitely on its Floyd boundary.
Then by Theorem \ref{mainthm4}, there is a parabolic subgroup $H \in
\mathbb H$ such that the limit set $\LF(H)$ is nontrivial and the
action $H \curvearrowright \LF(H)$ is not geometrically finite. By
Theorem C in \cite{GePo3}, $\LF(H)$ is homeomorphic to the Floyd
boundary of $H$. Therefore, this contradicts to Conjecture A.

\textbf{Conjecture B implies Conjecture A}: Suppose, to the
contrary, that there exists a NRH group $\Gamma$ with non-trivial
Floyd boundary. Then we take a free product $G = \Gamma
* F_2$, where $F_2$ is a free group of rank 2. By Conjecture
A, $G$ acts geometrically finitely on $\pGf$. By Theorem
\ref{mainthm4}, we have that $\Gamma$ also acts geometrically
finitely on $\LF(\Gamma)$. Using again Theorem C in \cite{GePo3}, we
obtain that $\Gamma$ acts geometrically finitely on its non-trivial
Floyd boundary. This contradicts to the hypothesis that $\Gamma$ is
a non-relatively hyperbolic group.
\end{proof}



\bibliographystyle{amsplain}

\begin{thebibliography}{10}
\bibitem{AAS} J.W. Anderson, J. Aramayona and K.J. Shackleton, \textit{An obstruction to the strong relative hyperbolicity of
a group}, J. Group Theory 10 (2007), no. 6, 749-756.

\bibitem{BDM} J. Behrstock, C. Drutu, L. Mosher, \textit{Thick metric spaces, relative hyperbolicity, and quasi-isometric rigidity}, Math. Annalen, 344, 2009, 543-595.

\bibitem{Bow1} B. Bowditch, \textit{Relatively hyperbolic groups}, Preprint, Univ. of Southampton,1999.

\bibitem{Bow2} B. Bowditch, \textit{Convergence groups and configuration spaces}, in "Group Theory Down Under" (J. Cossey, C.F. Miller, W.D. Neumann, M. Shapiro, eds.), de Gruyter (1999), 23-54.

\bibitem{Dah2} F. Dahmani. \textit{Accidental parabolics and relatively hyperbolic groups}. Israel Journal of Mathematics, 153(1): 93-127, 2006.

\bibitem{DruSapir} C. Drutu and M. Sapir. \textit{Tree-graded spaces and asymptotic cones of groups}. With an appendix by D. Osin and M. Sapir. Topology, 44(5): 959-1058, 2005.

\bibitem{Dun}  M. Dunwoody. \textit{An inaccessible group}. In Geometric group theory, Vol. 1 (Sussex, 1991), volume 181 of London Math. Soc. Lecture Note Ser., pages 75-78. Cambridge Univ. Press, Cambridge, 1993.

\bibitem{Farb} B. Farb, \textit{Relatively hyperbolic groups}, Geom. Funct. Anal., 8(5) (1998), 810-840.

\bibitem{Floyd} W. Floyd, \textit{Group completions and limit sets of Kleinian groups}, Inventiones Math. 57 (1980), 205-218.

\bibitem{Ge1} V. Gerasimov, \textit{Expansive convergence groups are relatively hyperbolic}, Geom. Funct. Anal. (19):137-169, 2009

\bibitem{Ge2} V. Gerasimov, \textit{Floyd maps to the boundaries of relatively hyperbolic groups}, arXiv:1001.4482.

\bibitem{GePo1} V. Gerasimov and L. Potyagailo, \textit{Dynamical quasiconvexity in relatively hyperbolic groups} preprint 2009.

\bibitem{GePo2} V. Gerasimov and L. Potyagailo, \textit{Quasi-isometries and Floyd boundaries of relatively hyperbolic groups}. arXiv:0908.0705

\bibitem{GePo3} V. Gerasimov and L. Potyagailo, \textit{Horospherical geometry of relatively hyperbolic groups}.  arXiv:1008.3470

\bibitem{GePo4} V. Gerasimov and L. Potyagailo, \textit{Quasiconvexity in the relatively hyperbolic groups}.  arXiv:1103.1211

\bibitem{Ger} S. Gersten, \textit{Subgroups of small cancellation groups in dimension 2}, J. London Math. Soc. , 54 (1996), 261-283.

\bibitem{GH}  E. Ghys, P. de la Harpe, Eds., \textit{Sur les groupes hyperboliques d'apr$\grave{e}$s Mikhael Gromov, Progress in Math., 83, Birka$\ddot{u}$ser}, 1990.

\bibitem{Gro} M. Gromov, \textit{Hyperbolic groups}, from: Essays in group theory (S Gersten, editor), Springer, New York (1987),75-263.

\bibitem{Hru} G. Hruska, \textit{Relative hyperbolicity and relative quasiconvexity for countable groups}, Algebr. Geom. Topol. \textbf{10} (2010) 1807--1856.

\bibitem{HrWi} G. Hruska and D. Wise, \textit{Packing subgroups in relatively hyperbolic groups},  Geom. Topol.  13(4)  (2009),1945-1988.

\bibitem{Ka} A. Karlsson, \textit{Free subgroups of groups with non-trivial Floyd boundary}, Comm. Algebra, 31, (2003), 5361-5376.

\bibitem{MarWise} E. Martinez-Pedroza and D. Wise, \textit{Relative Quasiconvexity using Fine Hyperbolic Graphs}, arXiv:1009.3532

\bibitem{MarPed} E. Martinez-Pedroza, \textit{On Quasiconvexity and Relatively Hyperbolic Structures on Groups}, arXiv:0811.2384

\bibitem{MarPed2} E. Martinez-Pedroza, \textit{Combination of Quasiconvex Subgroups of Relatively Hyperbolic Groups}. Groups, Geometry, and Dynamics, 3 (2009), 317-342.


\bibitem{MMJ} Mahan MJ, \textit{Relative Rigidity, Quasiconvexity and C-Complexes}, arXiv:0704.1922

\bibitem{Ol} A. Olshanskii, \textit{Periodic quotients of hyperbolic groups}, Mat. Sbornik 182 (1991), 4, 543-567 (in Russian), English translation in Math. USSR Sbornik 72 (1992), 2, 519-541.

\bibitem{OOS} A. Olshanskii, D. Osin and M. Sapir, \textit{Lacunary hyperbolic groups.} With an appendix by Michael Kapovich and Bruce Kleiner. Geom. Topol. \textbf{13}  (2009),  no. 4, 2051--2140.

\bibitem{Osin} D. Osin, \textit{Relatively hyperbolic groups: intrinsic geometry, algebraic properties, and algorithmic problems}, Mem. Amer. Math. Soc., 179(843) 2006, 1-100.

\bibitem{Osin2} D. Osin. \textit{Elementary subgroups of relatively hyperbolic groups and bounded generation}, Internat. J. Algebra Comput., 16(1):99-118, 2006.

\bibitem{Tukia} P. Tukia, \textit{Conical limit points and uniform convergence groups}, J. Reine. Angew. Math. 501 (1998) 71-98.

\bibitem{Yaman} A. Yaman, \textit{A topological characterisation of relatively hyperbolic groups}, J. reine ang. Math. 566 (2004), 41-89.

\bibitem{Yang} W. Yang, \textit{Peripheral structures of relatively hyperbolic groups}, PhD Thesis, Universit\'e de Lille 1, 2011.

\end{thebibliography}

\end{document}